\documentclass{amsart}

\usepackage{amsmath,amssymb,amsthm}
\usepackage{mathtools}
\usepackage{indentfirst}
\usepackage{float}
\usepackage{longtable}
\usepackage{booktabs}
\usepackage{enumerate}
\usepackage[shortlabels]{enumitem}
\usepackage[hmarginratio=1:1]{geometry}
\usepackage{diagbox}
\usepackage{eqparbox}
\usepackage{makecell}
\usepackage{comment}
\usepackage{mathrsfs}

\theoremstyle{definition}
\newtheorem{theorem}{Theorem}[section]
\newtheorem{definition}[theorem]{Definition}
\newtheorem{proposition}[theorem]{Proposition}
\newtheorem{lemma}[theorem]{Lemma}

\newtheorem{remark}[theorem]{Remark}

\numberwithin{equation}{section}
\numberwithin{table}{section}

\allowdisplaybreaks[4]

\newcommand{\NN}{\mathbb{N}}
\newcommand{\ZZ}{\mathbb{Z}}
\newcommand{\QQ}{\mathbb{Q}}
\newcommand{\RR}{\mathbb{R}}
\renewcommand{\AA}{\mathbb{A}}

\newcommand{\CC}{\mathbb{C}}
\newcommand{\HH}{\mathbb{H}}
\newcommand{\class}{\text{cls}}

\newcommand{\genus}{\text{gen}}

\newcommand{\mass}{m}

\newcommand{\ord}{\text{ord}}

\newcommand{\legendre}[2]{%
  \genfrac{(}{)}{}{}{#1}{#2}%
  \if\relax\detokenize{}\relax\else_{\!}\fi
}
\newcommand{\lcm}{\text{lcm}}
\renewcommand{\pmod}[1]{\text{~(mod }#1\text{)}}
\newcommand{\imag}{\operatorname{Im}}
\newcommand{\real}{\text{Re}}

\newcommand{\SL}{\text{SL}}
\newcommand{\dual}{\#}
\newcommand{\shiftinverse}[2]{#1(#2)}
\newcommand{\td}{\mathfrak{t}_d}

\newcommand{\Constant}{\Gamma}

\newif\ifdefs
\defstrue

\newif\ifdiscs
\discstrue

\ifdiscs
\else
\excludecomment{discussion}
\fi

\ifdefs
\else
\excludecomment{extradetailsproof}
\fi

\newif\ifold
\oldfalse

\newif\ifnew
\newtrue

\title{On finiteness theorems for sums of generalized polygonal numbers}
\author{Ben Kane}
\address{Department of Mathematics, University of Hong Kong, Pokfulam, Hong Kong}
\email{bkane@hku.hk}
\author{Zichen Yang}
\address{Department of Mathematics, University of Hong Kong, Pokfulam, Hong Kong}
\email{zichenyang.math@gmail.com}
\date{\today}
\thanks{The first author was supported by grants from the Research Grants Council of the Hong Kong SAR, China (project numbers HKU 17303618 and 17314122).}
\begin{document}
\begin{abstract}
In this paper, we consider mixed sums of generalized polygonal numbers. Specifically, we obtain a finiteness condition for universality of such sums; this means that it suffices to check representability of a finite subset of the positive integers in order to conclude that the sum of generalized polygonal numbers represents every positive integer. The sub-class of sums of generalized polygonal numbers which we consider is those sums of $m_j$-gonal numbers for which $\lcm(m_1-2,\dots,m_{r}-2)\leq \mathfrak{M}$ and we obtain a bound on the asymptotic growth of a constant $\Constant_{\mathfrak{M}}$ such that it suffices to check the representability condition for $n\leq \Constant_{\mathfrak{M}}$.
\end{abstract}
\maketitle

\section{Introduction}

A milestone in the arithmetic theory of quadratic forms is the Conway--Schneeberger $15$-theorem \cite{conway1999universal}, which states that a positive-definite integral quadratic form represents every positive integer if and only if it represents every positive integer up to $15$. Besides its elegant statement, it provides a satisfactory classification of \begin{it}universal\end{it} positive-definite integral quadratic forms, which are quadratic forms representing every positive integer. A theorem of this type is usually called a finiteness theorem.

Motivated by this theorem, it is natural to generalize the finiteness result to quadratic polynomials. In particular, we are interested in sums of generalized polygonal numbers. For an integer $m\geq3$ and $x\in\ZZ$, we define the $x$-th \begin{it}generalized $m$-gonal number\end{it} to be
$$
P_m(x)\coloneqq\frac{(m-2)x^2-(m-4)x}{2}.
$$
A \textit{sum of generalized polygonal numbers} is a function $F\colon\ZZ^r\to\ZZ$ of the form 
$$
F(x_1,\ldots,x_r)=a_1P_{m_1}(x_1)+\cdots+a_rP_{m_r}(x_r)
$$
with integer parameters $1\leq a_1\leq\cdots\leq a_r$ and $m_1,\ldots,m_r\geq3$. If $m_1=\cdots=m_r=m$ for some integer $m\geq3$, the sum $F$ is called a \textit{sum of $m$-gonal numbers}. We say that a positive integer $n$ is represented by a sum $F$, if there exist $x_1,\ldots,x_r\in\ZZ$ such that $n=F(x_1,\ldots,x_r)$. Such a tuple $(x_1,\ldots,x_r)\in\ZZ^r$ is called a \textit{representation} of $n$ by $F$. We denote $r_F(n)$ the number of representations of $n$ by $F$. If every positive integer is represented by $F$, we say that $F$ is \textit{universal}.

A number of studies are devoted to prove finiteness theorems for sums of $m$-gonal numbers. Indeed, by constructing Bhargava's escalator tree \cite{bhargava2000}, one can prove that for any integer $m\geq3$, there exists a minimal constant $\gamma_m>0$ such that every sum of $m$-gonal numbers is universal if and only if it represents every positive integers up to $\gamma_m$ \cite[Lemma 2.1]{kane2020universal}. For small values of the parameter $m$, the constant $\gamma_m$ has been determined. To name a few, Bosma and the first author \cite{bosma2013triangular} showed that $\gamma_3=\gamma_6=8$. The next case $\gamma_4=15$ is a consequence of Conway--Schneeberger's $15$-theorem. For $m=5$, Ju \cite{ju2020universal} showed that $\gamma_5=109$. For $m=7$, Kamaraj, the first author, and Tomiyasu provided an upper bound for $\gamma_7$ in \cite{kamaraj2022universal}. For $m=8$, Ju and Oh \cite{ju2018universal} proved that $\gamma_8=60$. In a recent paper, the first author and Liu \cite[Theorem 1.1(1)]{kane2020universal} studied the growth of $\gamma_m$ as $m\to\infty$. More precisely, they proved that for any real number $\varepsilon>0$, we have
\begin{equation}
\label{eqn::previousbound}
m-4\leq\gamma_m\ll_{\varepsilon}m^{7+\varepsilon}
\end{equation}
where the implied constant is effective. Later, Kim and Park \cite{kim2020finiteness} improved the upper bound by showing that the growth of $\gamma_m$ is at most linear, giving lower and upper bounds for $\gamma_m$ that are both linear, although the constant or proportionality for the upper bound is not explicitly known.

Generally, for a class $\mathcal{C}$ of sums of generalized polygonal numbers, one can construct the corresponding escalator tree $T_{\mathcal{C}}$ and try to prove a finiteness theorem with a constant $\Constant_{\mathcal{C}}>0$ depending on the class $\mathcal{C}$. For example, the classes considered above are those consisting of sums of $m$-gonal numbers for a fixed integer $m\geq3$. The most general class one could consider is the class $\mathcal{C}_{\infty}$ containing all sums of generalized polygonal numbers with arbitrary parameters $m_1\ldots,m_r\geq3$. However, such a uniform constant $\Constant_{\infty}$ does not exist because by (\ref{eqn::previousbound}) we have an arbitrarily large lower bound $\Constant_{\infty}\geq\gamma_m\geq m-4$.

It is thus natural to consider subclasses of $\mathcal{C}_{\infty}$. For any integer $\mathfrak{M}\geq1$, we consider the class $\mathcal{C}_{\mathfrak{M}}$ consisting of sums of generalized polygonal numbers with parameters $m_1,\ldots,m_r\geq3$ such that $\lcm(m_1-2,\ldots,m_r-2)\leq\mathfrak{M}$. The reason we consider $\lcm(m_1-2,\ldots,m_r-2)$ rather than an a priori more natural choice $\lcm(m_1,\ldots,m_r)$ will become apparent later on. For these classes, we have the following finiteness theorems.

\begin{theorem}
\label{thm::mainfinitenesstheorem1}
For any integer $\mathfrak{M}\geq1$, there exists a minimal constant $\Constant_{\mathfrak{M}}>0$ such that every sum of generalized polygonal numbers in $\mathcal{C}_{\mathfrak{M}}$ is universal if and only if it represents every positive integer up to $\Constant_{\mathfrak{M}}$. Moreover, for any real number $\varepsilon>0$, we have
$$
\Constant_{\mathfrak{M}}\ll_{\varepsilon}\mathfrak{M}^{43+\varepsilon},
$$
where the implied constant is ineffective.
\end{theorem}

\begin{remark}
\label{rmk::ineffective}
With the current technology, the ineffectiveness is inevitable because there are finitely many ternary sums appearing in the escalator tree that are highly likely to be universal sums but we have no idea to prove the universality of all of them. The technical difficulty was extensively summarized in \cite{ono1997ramanujan}. There are also potentially universal quaternary sums in the escalator tree. They are easier to deal with but we do not work out the details for these because the ineffectiveness would remain due to the lack of classification of universal ternary sums. Furthermore, it is worth noting that an additional difficulty when dealing with quaternary sums is the existence of infinitely many universal quaternary sums in the escalator tree. For example, assuming the truth of various generalized Riemann hypotheses, the first author proved that every positive integer except $23$ is represented by $P_4+2P_4+3P_3$ in \cite[Theorem 1.8]{kane2009two}. Therefore any child of $P_4+2P_4+3P_3$ in the escalator tree is universal. For details, see the construction of the escalator tree below.
\end{remark}

\begin{remark}
Since $\mathcal{C}_{1}$ is exactly the class of sums of triangular numbers, we know that $\Constant_{1}=8$. The class $\mathcal{C}_{2}$ is the class of mixed sums of triangular numbers and squares. Because ternary sums in $\mathcal{C}_{2}$ were classified in \cite{sun2007mixed,guo2007mixed,oh2009mixed}, the constant $\Constant_{2}$ is effective in principle. The second author is going to make it explicit in a forthcoming paper. The constant $\Constant_{\mathfrak{M}}$ is ineffective for any integer $\mathfrak{M}\geq3$.
\end{remark}

Remark \ref{rmk::ineffective} leads to the following conditional result.

\begin{theorem}
\label{thm::mainfinitenesstheorem2}
Assuming that the classification of universal ternary sums has been completed, the implied constants in Theorem \ref{thm::mainfinitenesstheorem1} can be made effective.
\end{theorem}

In fact, a conjectural list of $3052$ universal ternary sums can be determined via a computer search, where sums involving $P_6$ are excluded because $P_3$ and $P_6$ represent the same set of integers. Although it seems impossible to prove the universality of all of them, there are $197$ ternary sums among them that are confirmed to be universal, by a non-exhaustive search of relevant literature and by algebraic methods. The confirmed cases are listed in Table \ref{tbl::confirmedcases} and the complete data of the conjectural list can be found in the thesis of the second author.

{\footnotesize\setlength{\tabcolsep}{1pt}
\begin{longtable}{llllllllll}
\caption{The list of ternary sums that are confirmed to be universal.}
\label{tbl::confirmedcases}\\
\toprule
$P_{3}+P_{3}+P_{3}$: &\cite{liouville1863nouveaux} & $P_{3}+P_{3}+P_{4}$: &\cite{sun2007mixed} & $P_{3}+P_{3}+P_{5}$: &\cite{sun2015universal} & $P_{3}+P_{3}+P_{7}$: &\cite{sun2015universal} & $P_{3}+P_{3}+P_{8}$: &\cite{sun2015universal}\\
$P_{3}+P_{3}+P_{10}$: &\cite{sun2015universal} & $P_{3}+P_{3}+P_{11}$ & & $P_{3}+P_{3}+P_{12}$: &\cite{ju2019ternary} & $P_{3}+P_{3}+2P_{3}$: &\cite{liouville1863nouveaux} & $P_{3}+P_{3}+2P_{4}$: &\cite{sun2007mixed}\\
$P_{3}+P_{3}+2P_{5}$: &\cite{sun2015universal} & $P_{3}+P_{3}+2P_{8}$ & & $P_{3}+P_{3}+3P_{5}$ & & $P_{3}+P_{3}+4P_{3}$: &\cite{liouville1863nouveaux} & $P_{3}+P_{3}+4P_{4}$: &\cite{sun2007mixed}\\
$P_{3}+P_{3}+4P_{5}$: &\cite{sun2015universal} & $P_{3}+P_{3}+5P_{3}$: &\cite{liouville1863nouveaux} & $P_{3}+P_{4}+P_{4}$: &\cite{sun2007mixed} & $P_{3}+P_{4}+P_{5}$: &\cite{sun2015universal} & $P_{3}+P_{4}+P_{7}$: &\cite{sun2015universal}\\
$P_{3}+P_{4}+P_{8}$: &\cite{sun2015universal} & $P_{3}+P_{4}+P_{9}$: &\cite{ju2019ternary} & $P_{3}+P_{4}+P_{10}$: &\cite{sun2015universal} & $P_{3}+P_{4}+P_{11}$: &\cite{sun2015universal} & $P_{3}+P_{4}+P_{12}$: &\cite{sun2015universal}\\
$P_{3}+P_{4}+2P_{3}$: &\cite{sun2007mixed} & $P_{3}+P_{4}+2P_{4}$: &\cite{sun2007mixed} & $P_{3}+P_{4}+2P_{5}$: &\cite{sun2015universal} & $P_{3}+P_{4}+2P_{7}$: &\cite{ju2019ternary} & $P_{3}+P_{4}+2P_{8}$: &\cite{sun2015universal}\\
$P_{3}+P_{4}+3P_{3}$: &\cite{guo2007mixed} & $P_{3}+P_{4}+3P_{4}$: &\cite{guo2007mixed} & $P_{3}+P_{4}+4P_{3}$: &\cite{sun2007mixed} & $P_{3}+P_{4}+4P_{4}$: &\cite{sun2007mixed} & $P_{3}+P_{4}+6P_{3}$: &\cite{guo2007mixed}\\
$P_{3}+P_{4}+8P_{3}$: &\cite{oh2009mixed} & $P_{3}+P_{4}+8P_{4}$: &\cite{sun2007mixed} & $P_{3}+P_{5}+P_{5}$: &\cite{sun2015universal} & $P_{3}+P_{5}+P_{7}$: &\cite{ju2019ternary} & $P_{3}+P_{5}+P_{8}$: &\cite{sun2015universal}\\
$P_{3}+P_{5}+P_{9}$: &\cite{sun2015universal} & $P_{3}+P_{5}+P_{11}$: &\cite{ge2016some} & $P_{3}+P_{5}+P_{12}$ & & $P_{3}+P_{5}+P_{13}$: &\cite{ju2019ternary} & $P_{3}+P_{5}+2P_{3}$: &\cite{sun2015universal}\\
$P_{3}+P_{5}+2P_{4}$: &\cite{sun2015universal} & $P_{3}+P_{5}+2P_{5}$ & & $P_{3}+P_{5}+2P_{7}$: &\cite{sun2015universal} & $P_{3}+P_{5}+2P_{8}$ & & $P_{3}+P_{5}+2P_{9}$: &\cite{ju2019ternary}\\
$P_{3}+P_{5}+3P_{3}$: &\cite{sun2015universal} & $P_{3}+P_{5}+3P_{4}$: &\cite{sun2015universal} & $P_{3}+P_{5}+4P_{3}$: &\cite{sun2015universal} & $P_{3}+P_{5}+4P_{4}$: &\cite{sun2015universal} & $P_{3}+P_{5}+4P_{5}$ &\\
$P_{3}+P_{5}+6P_{3}$: &\cite{sun2015universal} & $P_{3}+P_{5}+9P_{3}$: &\cite{ju2019ternary} & $P_{3}+P_{7}+P_{7}$: &\cite{sun2015universal} & $P_{3}+P_{7}+P_{8}$: &\cite{ju2019ternary} & $P_{3}+P_{7}+P_{10}$: &\cite{sun2015universal}\\
$P_{3}+P_{7}+2P_{3}$: &\cite{ju2019ternary} & $P_{3}+P_{7}+2P_{5}$: &\cite{ju2019ternary} & $P_{3}+P_{7}+2P_{7}$: &\cite{ju2019ternary} & $P_{3}+P_{7}+5P_{3}$ & & $P_{3}+P_{8}+P_{8}$ &\\
$P_{3}+P_{8}+2P_{3}$: &\cite{sun2015universal} & $P_{3}+P_{8}+2P_{4}$: &\cite{sun2015universal} & $P_{3}+P_{8}+2P_{5}$ & & $P_{3}+P_{8}+3P_{3}$ & & $P_{3}+P_{8}+3P_{4}$: &\cite{sun2015universal}\\
$P_{3}+P_{8}+4P_{3}$ & & $P_{3}+P_{8}+6P_{3}$: &\cite{sun2015universal} & $P_{3}+P_{9}+2P_{3}$: &\cite{ju2019ternary} & $P_{3}+P_{10}+2P_{3}$: &\cite{sun2015universal} & $P_{3}+P_{12}+2P_{3}$: &\cite{ju2019ternary}\\
$P_{3}+2P_{3}+2P_{3}$: &\cite{liouville1863nouveaux} & $P_{3}+2P_{3}+2P_{4}$: &\cite{sun2007mixed} & $P_{3}+2P_{3}+2P_{7}$: &\cite{sun2015universal} & $P_{3}+2P_{3}+2P_{8}$: &\cite{ju2019ternary} & $P_{3}+2P_{3}+3P_{3}$: &\cite{liouville1863nouveaux}\\
$P_{3}+2P_{3}+3P_{4}$: &\cite{guo2007mixed} & $P_{3}+2P_{3}+4P_{3}$: &\cite{liouville1863nouveaux} & $P_{3}+2P_{3}+4P_{4}$: &\cite{guo2007mixed} & $P_{3}+2P_{3}+4P_{5}$: &\cite{sun2015universal} & $P_{3}+2P_{4}+2P_{4}$: &\cite{sun2007mixed}\\
$P_{3}+2P_{4}+2P_{5}$: &\cite{sun2015universal} & $P_{3}+2P_{4}+4P_{3}$: &\cite{sun2007mixed} & $P_{3}+2P_{4}+4P_{5}$: &\cite{ju2019ternary} & $P_{3}+2P_{5}+2P_{5}$ & & $P_{3}+2P_{5}+3P_{3}$ &\\
$P_{3}+2P_{5}+3P_{4}$: &\cite{sun2015universal} & $P_{3}+2P_{5}+4P_{3}$: &\cite{sun2015universal} & $P_{3}+2P_{5}+4P_{4}$: &\cite{ju2019ternary} & $P_{3}+2P_{5}+6P_{3}$ & & $P_{3}+2P_{8}+4P_{5}$ &\\
$P_{4}+P_{4}+P_{5}$: &\cite{sun2015universal} & $P_{4}+P_{4}+P_{8}$: &\cite{sun2015universal} & $P_{4}+P_{4}+P_{10}$: &\cite{sun2015universal} & $P_{4}+P_{4}+2P_{3}$: &\cite{sun2007mixed} & $P_{4}+P_{4}+2P_{5}$ &\\
$P_{4}+P_{4}+2P_{8}$: &\cite{ju2020universal} & $P_{4}+P_{5}+P_{5}$: &\cite{sun2015universal} & $P_{4}+P_{5}+P_{10}$ & & $P_{4}+P_{5}+2P_{3}$ & & $P_{4}+P_{5}+2P_{4}$: &\cite{sun2015universal}\\
$P_{4}+P_{5}+2P_{5}$ & & $P_{4}+P_{5}+2P_{8}$ & & $P_{4}+P_{5}+3P_{3}$: &\cite{sun2015universal} & $P_{4}+P_{5}+4P_{5}$ & & $P_{4}+P_{5}+6P_{3}$ &\\
$P_{4}+P_{7}+2P_{3}$ & & $P_{4}+P_{7}+2P_{5}$ & & $P_{4}+P_{8}+P_{8}$: &\cite{sun2015universal} & $P_{4}+P_{8}+2P_{3}$ & & $P_{4}+P_{8}+2P_{5}$ &\\
$P_{4}+P_{8}+3P_{4}$: &\cite{ju2020universal} & $P_{4}+P_{10}+2P_{3}$ & & $P_{4}+P_{12}+2P_{4}$ & & $P_{4}+2P_{3}+2P_{3}$: &\cite{sun2007mixed} & $P_{4}+2P_{3}+2P_{4}$: &\cite{sun2007mixed}\\
$P_{4}+2P_{3}+2P_{5}$ & & $P_{4}+2P_{3}+2P_{8}$ & & $P_{4}+2P_{3}+4P_{3}$: &\cite{sun2007mixed} & $P_{4}+2P_{3}+4P_{4}$: &\cite{sun2007mixed} & $P_{4}+2P_{3}+4P_{5}$ &\\
$P_{4}+2P_{3}+5P_{3}$: &\cite{sun2007mixed} & $P_{4}+2P_{4}+4P_{3}$: &\cite{sun2007mixed} & $P_{4}+2P_{5}+2P_{5}$ & & $P_{4}+2P_{5}+3P_{3}$ & & $P_{4}+2P_{5}+6P_{3}$ &\\
$P_{4}+2P_{8}+4P_{5}$ & & $P_{5}+P_{5}+P_{5}$: &\cite{guy1994every} & $P_{5}+P_{5}+P_{10}$: &\cite{sun2015universal} & $P_{5}+P_{5}+2P_{3}$ & & $P_{5}+P_{5}+2P_{4}$ &\\
$P_{5}+P_{5}+2P_{5}$: &\cite{sun2015universal} & $P_{5}+P_{5}+2P_{8}$ & & $P_{5}+P_{5}+3P_{3}$ & & $P_{5}+P_{5}+3P_{5}$: &\cite{ge2016some} & $P_{5}+P_{5}+4P_{3}$ &\\
$P_{5}+P_{5}+4P_{5}$: &\cite{sun2015universal} & $P_{5}+P_{5}+5P_{5}$: &\cite{sun2015universal} & $P_{5}+P_{5}+6P_{5}$: &\cite{oh2011ternary} & $P_{5}+P_{5}+8P_{5}$: &\cite{oh2011ternary} & $P_{5}+P_{5}+9P_{5}$: &\cite{oh2011ternary}\\
$P_{5}+P_{5}+10P_{5}$: &\cite{oh2011ternary} & $P_{5}+P_{7}+3P_{3}$: &\cite{ge2016some} & $P_{5}+P_{8}+2P_{3}$ & & $P_{5}+P_{8}+2P_{5}$ & & $P_{5}+P_{8}+3P_{3}$ &\\
$P_{5}+P_{8}+3P_{4}$ & & $P_{5}+P_{9}+2P_{3}$: &\cite{ju2019ternary} & $P_{5}+P_{9}+3P_{3}$ & & $P_{5}+P_{10}+2P_{3}$ & & $P_{5}+P_{10}+3P_{3}$ &\\
$P_{5}+2P_{3}+2P_{3}$ & & $P_{5}+2P_{3}+2P_{8}$ & & $P_{5}+2P_{3}+3P_{3}$: &\cite{sun2015universal} & $P_{5}+2P_{3}+3P_{4}$: &\cite{sun2015universal} & $P_{5}+2P_{3}+4P_{5}$ &\\
$P_{5}+2P_{4}+3P_{3}$ & & $P_{5}+2P_{4}+4P_{3}$ & & $P_{5}+2P_{4}+6P_{4}$: &\cite{sun2015universal} & $P_{5}+2P_{5}+2P_{5}$: &\cite{sun2015universal} & $P_{5}+2P_{5}+3P_{3}$ &\\
$P_{5}+2P_{5}+3P_{5}$: &\cite{ge2016some} & $P_{5}+2P_{5}+4P_{5}$: &\cite{sun2015universal} & $P_{5}+2P_{5}+6P_{3}$ & & $P_{5}+2P_{5}+6P_{5}$: &\cite{ge2016some} & $P_{5}+2P_{5}+8P_{5}$: &\cite{oh2011ternary}\\
$P_{5}+2P_{8}+4P_{5}$ & & $P_{5}+3P_{3}+3P_{3}$ & & $P_{5}+3P_{3}+3P_{4}$: &\cite{sun2015universal} & $P_{5}+3P_{3}+4P_{3}$ & & $P_{5}+3P_{3}+5P_{5}$ &\\
$P_{5}+3P_{4}+6P_{3}$ & & $P_{5}+3P_{5}+3P_{5}$: &\cite{ge2016some} & $P_{5}+3P_{5}+4P_{5}$: &\cite{ge2016some} & $P_{5}+3P_{5}+6P_{5}$: &\cite{sun2015universal} & $P_{5}+3P_{5}+7P_{5}$: &\cite{oh2011ternary}\\
$P_{5}+3P_{5}+8P_{5}$: &\cite{oh2011ternary} & $P_{5}+3P_{5}+9P_{5}$: &\cite{ge2016some} & $P_{7}+P_{8}+2P_{3}$ & & $P_{8}+P_{8}+2P_{3}$ & & $P_{8}+P_{8}+2P_{4}$: &\cite{ju2020universal}\\
$P_{8}+P_{8}+2P_{5}$ & & $P_{8}+P_{10}+2P_{5}$ & & $P_{8}+2P_{3}+2P_{3}$ & & $P_{8}+2P_{3}+2P_{5}$ & & $P_{8}+2P_{3}+3P_{3}$ &\\
$P_{8}+2P_{3}+3P_{4}$: &\cite{sun2015universal} & $P_{8}+2P_{4}+2P_{5}$ & & $P_{8}+2P_{4}+3P_{4}$: &\cite{ju2020universal} & $P_{8}+2P_{4}+4P_{3}$ & & $P_{8}+2P_{5}+2P_{5}$ &\\
$P_{8}+2P_{5}+3P_{3}$ & & $P_{8}+2P_{5}+4P_{3}$ & & & & & & &\\
\bottomrule
\end{longtable}}

\begin{theorem}
\label{thm::ternaryuniversal}
Excluding sums involving generalized hexagonal numbers, there are at most $3052$ ternary universal sums. Among them, there are $197$ ternary sums that are confirmed to be universal, given in Table \ref{tbl::confirmedcases}.
\end{theorem}

We can avoid the existence of universal ternary and quaternary sums to obtain effective results, as well as better growth by restricting to subclasses of $\mathcal{C}_{\mathfrak{M}}$. For any integers $\mathfrak{m},\mathfrak{M}\in\ZZ$ such that $\mathfrak{M}\geq\mathfrak{m}-2\geq1$, we define the subclass $\mathcal{C}_{\mathfrak{m},\mathfrak{M}}$ as the class of sums of generalized polygonal numbers with parameters $m_1,\ldots,m_r\geq3$ such that $\lcm(m_1-2,\ldots,m_r-2)\leq\mathfrak{M}$ and $\min(m_1,\ldots,m_r)\geq\mathfrak{m}$. For these classes, we can establish effective finiteness theorems.

\begin{theorem}
\label{thm::mainfinitenesstheorem3}
For any integers $\mathfrak{m},\mathfrak{M}\in\ZZ$ such that $\mathfrak{M}\geq\mathfrak{m}-2\geq1$, there exists a minimal constant $\Constant_{\mathfrak{m},\mathfrak{M}}>0$ such that every sum of generalized polygonal numbers in $\mathcal{C}_{\mathfrak{m},\mathfrak{M}}$ is universal if and only if it represents every positive integer up to $\Constant_{\mathfrak{m},\mathfrak{M}}$. For any real number $\varepsilon>0$, we have the following upper bounds for the constant $\Constant_{\mathfrak{m},\mathfrak{M}}$:
\begin{enumerate}[leftmargin=*]
\item If $\mathfrak{m}\geq20$, then we have
$$
\Constant_{\mathfrak{m},\mathfrak{M}}\ll_{\varepsilon}\mathfrak{M}^{43+\varepsilon}.
$$
\item If $\mathfrak{m}\geq36$, then we have
$$
\Constant_{\mathfrak{m},\mathfrak{M}}\ll_{\varepsilon}\mathfrak{M}^{27+\varepsilon}.
$$
\end{enumerate}
For each case, the implied constant is effective.
\end{theorem}

The rest of the paper is organized as follows. We begin by constructing the escalator tree and proving elementary properties of the escalator tree in Section \ref{sec::escalator}. To obtain the existence and the growth of the constants $\Constant_{\mathfrak{M}}$ and $\Constant_{\mathfrak{m},\mathfrak{M}}$, we have to study the representations of integers by nodes. In Section \ref{sec::shiftedlattice}, we convert the problem into the study of the number of representations by a shifted lattice, which splits as a sum of the Fourier coefficients of an Eisenstein series and a cusp form. In Section \ref{sec::Eisenstein} and Section \ref{sec::cuspidal}, we give estimates on the Eisenstein part and the cuspidal part, respectively. Finally, we prove the main theorem.

\section{Elementary Properties of the Escalator Tree}\label{sec::escalator}

We first construct the escalator tree for a given class $\mathcal{C}$ of sums of generalized polygonal numbers. The escalator tree $T_{\mathcal{C}}$ for universal sums of generalized polygonal numbers in $\mathcal{C}$ is a rooted tree constructed inductively as follows. We define the root to be $F=0$ with depth $0$ and then inductively construct the nodes of depth $r+1$ from the nodes of depth $r$ as follows. If a node of depth $r$ is universal, then it is a leaf of the tree. If a sum of generalized polygonal numbers $F$ is not universal, then we call the smallest positive integer $t(F)$ not represented by $F$ the \begin{it}truant\end{it} of $F$; for ease of notation, we write $t(F)\coloneqq\infty$ if $F$ is universal. The children of a node $F(x_1,\ldots,x_r)=\sum_{j=1}^r a_j P_{m_j}(x_j)$ with $t(F)<\infty$ are the sums of generalized polygonal numbers
\[
F(x_1,\cdots,x_r)+a_{r+1}P_{m_{r+1}}(x_{r+1})\in\mathcal{C}
\]
with $a_{r}\leq a_{r+1}\leq t(F)$, $m_{r+1}\neq6$, and an additional restriction that $m_{r}\leq m_{r+1}$ if $a_{r}=a_{r+1}$ to avoid repeated nodes. For any class $\mathcal{C}$, we define the set 
$$
\mathscr{T}_{\mathcal{C}}\coloneqq\left\{t(F)<\infty\mid F\in T_{\mathcal{C}}\right\}.
$$
By construction, this subset has the property that if a sum $F\in\mathcal{C}$ represents every integer in $\mathscr{T}_{\mathcal{C}}$, then $F$ is universal. We also see that $\mathscr{T}_{\mathcal{C}}$ is minimal in the sense that it is the smallest subset of $\NN$ with this property. Therefore, if $\mathscr{T}_{\mathcal{C}}$ is a finite set, we obtain a finiteness theorem by taking $\Gamma_{\mathcal{C}}$ to be the maximal integer contained in $\mathscr{T}_{\mathcal{C}}$. 

We denote $T_{\infty}$ the escalator tree for the class $\mathcal{C}_{\infty}$ of arbitrary sums of generalized polygonal numbers. For any integer $\mathfrak{M}\geq1$, we abbreviate $T_{\mathfrak{M}}$ for the tree $T_{\mathcal{C}_{\mathfrak{M}}}$. We will show in Theorem \ref{thm::finitetree} that $T_{\mathfrak{M}}$ is a finite tree. Therefore, the existence of the constant $\Constant_{\mathfrak{M}}>0$ in Theorem \ref{thm::mainfinitenesstheorem1} is established. Similarly, we abbreviate $T_{\mathfrak{m},\mathfrak{M}}$ for the tree $T_{\mathcal{C}_{\mathfrak{m},\mathfrak{M}}}$ for any integers $\mathfrak{m},\mathfrak{M}\in\ZZ$ such that $\mathfrak{M}\geq\mathfrak{m}-2\geq1$. Then the existence of the constant $\Constant_{\mathfrak{m},\mathfrak{M}}$ in Theorem \ref{thm::mainfinitenesstheorem3} follows immediately because $\Constant_{\mathfrak{m},\mathfrak{M}}\leq\Constant_{\mathfrak{M}}$. To study the growth of the constants $\Constant_{\mathfrak{M}}$ and $\Constant_{\mathfrak{m},\mathfrak{M}}$, it is clear that we have to study the truants of non-leaf nodes in the escalator tree $T_{\infty}$. 

We begin with an observation on the truants of sums. 
\begin{lemma}\label{thm::observation}
Fix $1\leq i\leq r$. Let $F'=a_1P_{m_{1}}+\cdots+a_iP_{m_i'}+\cdots+a_rP_{m_{r}}$ be a sum with $t(F')<\infty$. For any integer $m_i\geq4+\lfloor t(F')/a_i\rfloor$, the sum $F=a_1P_{m_{1}}+\cdots+a_iP_{m_i}+\cdots+a_rP_{m_{r}}$ has truant $t(F)=c\leq t(F')$, where $c$ is a constant depending on $a_j$ with $1\leq j\leq r$ and $m_{j}$ with $j\neq i$.
\end{lemma}
\begin{proof}
If $t(F)>t(F')$ for some $m_i\geq 4+\lfloor t(F')/a_i\rfloor$, then there exists $x_1,\ldots,x_r\in\ZZ$ such that
$$
a_1P_{m_{1}}(x_1)+\cdots+a_iP_{m_i}(x_i)+\cdots+a_rP_{m_{r}}(x_r)=t(F').
$$
Since $F'$ cannot represent $t(F')$, we see that $x_i\neq0,1$. Noticing that $m_i\geq4$, we have $P_{m_i}(x_i)\geq m_i-3$. Therefore, we have
$$
a_1P_{m_{1}}(x_1)+\cdots+a_iP_{m_i}(x_i)+\cdots+a_rP_{m_{r}}(x_r)\geq a_iP_{m_i}(x_i)\geq a_i(1+\lfloor t(F')/a_i\rfloor)>t(F'),
$$
which is a contradiction. This shows that $t(F)\leq t(F')$. For the same reason, the representations of integers $n\leq t(F')-1$ are independent of the choice of $m_i$. Thus, the truant $t(F)$ is a constant $c$ for any integer $m_i\geq4+\lfloor t(F')/a_i\rfloor$.
\end{proof}


As a consequence, we obtain the following useful enumeration of the nodes in $T_{\infty}$ of depth $r\leq3$.


\begin{proposition}
\label{thm::truant2}
Suppose that $a_1P_{m_1}+a_2P_{m_2}$ is a node in $T_{\infty}$. Then, $a_1=1$ and $1\leq a_2\leq 3$. For each choice of $a_2$, the truants are summarized in Table \ref{tbl::truant2}.
\end{proposition}
\begin{proof}
One can calculate the truants of nodes with $m_1,m_2\leq10$. The calculation of the truants of other nodes follows from by Lemma \ref{thm::observation}. For example, to show that $t(P_3+P_{m_2})=5$ for any integer $m_2\geq9$, one applies Lemma \ref{thm::observation} with $i=2$ and $F'=P_3+P_6$.
\end{proof}

\begin{table}[htb]
\caption{Truants of nodes of depth $2$}
\begin{tabular}{|c|p{0.4cm}<{\centering}|p{0.4cm}<{\centering}|p{0.4cm}<{\centering}|p{0.4cm}<{\centering}|p{0.4cm}<{\centering}|c|p{0.4cm}<{\centering}|p{0.4cm}<{\centering}|p{0.4cm}<{\centering}|p{0.4cm}<{\centering}|p{0.4cm}<{\centering}|p{0.4cm}<{\centering}|c|c|}
\cline{2-15}
\multicolumn{1}{c|}{} & \multicolumn{6}{c|}{$a_2=1$} & \multicolumn{7}{c|}{$a_2=2$} & $a_2=3$\\
\hline
\diagbox{$m_2$}{$m_1$} & 3 & 4 & 5 & 7 & 8 & $\geq9$ & 3 & 4 & 5 & 7 & 8 & 9 & $\geq10$ & 5\\
\hline
3        & 5  & 8  & 9  & 9  & 12 & 5 & 4 & 5 & 10 & 5  & 4 & 4 & 4 & 6\\
\hline
4        & 8  & 3  & 20 & 3  & 3  & 3 & 4 & 5 & 6  & 5  & 4 & 4 & 4 & 6\\
\hline
5        & 9  & 20 & 11 & 10 & 4  & 4 & 9 & 7 & 8  & 12 & 6 & 7 & 6 & 9\\
\hline
7        & 9  & 3  & 10 & 3  & 3  & 3 & 4 & 5 & 6  & 5  & 4 & 4 & 4 & 6\\
\hline
8        & 12 & 3  & 4  & 3  & 3  & 3 & 4 & 5 & 6  & 5  & 4 & 4 & 4 & 6\\
\hline
$\geq9$  & 5  & 3  & 4  & 3  & 3  & 3 & 4 & 5 & 6  & 5  & 4 & 4 & 4 & 6\\
\hline
\end{tabular}
\label{tbl::truant2}
\end{table}

The following lemma is a generalization of Lemma \ref{thm::observation}, which is used to bound the truants of sums when varying multiple parameters.

\begin{lemma}
\label{thm::boundingtruant}
Fix $r\geq1$ and parameters $a_1,\ldots,a_r\in\ZZ$. Suppose that there exist integers $A_i$ and $B_i$ such that $B_i\geq A_i\geq3$ for any $1\leq i\leq r$ with the property: For any $1\leq i\leq r$ and for any integers $A_j\leq m_j\leq B_j$ for $1\leq j\neq i\leq r$, there exists a non-universal sum $F'=a_1P_{m_1}+\cdots+a_iP_{m_i'}+\cdots+a_rP_{m_r}$ with $A_i\leq m_i'\leq B_i$ such that $\left\lfloor t(F')/a_k\right\rfloor+3\leq B_k$ for any $1\leq k\leq r$. Then any sum $F$ of the form $a_1P_{m_1}+\cdots+a_rP_{m_r}$ with $m_i\geq A_i$ for any $1\leq i\leq r$ and $m_j\geq B_j+1$ for some $1\leq j\leq r$ has the truant
$$
t(F)\leq t_{\max}\coloneqq\max\{t(a_1P_{m_1}+\cdots+a_rP_{m_r})<\infty\mid A_i\leq m_i\leq B_i\text{ for }1\leq i\leq r\}.
$$
\end{lemma}
\begin{proof}
Let $F$ be any sum of the form $a_1P_{m_1}+\cdots+a_rP_{m_r}$ with $m_i\geq A_i$ for any $1\leq i\leq r$ and $m_j\geq B_j+1$ for some $1\leq j\leq r$. Let $s$ be the number of indices $1\leq j\leq r$ such that $m_j\geq B_j+1$. By induction on the integer $s\geq1$, we prove the lemma jointly with the following inequality
$$
\left\lfloor\frac{t(F)}{a_k}\right\rfloor+3\leq B_k
$$
for any $1\leq k\leq r$. If $s=1$, without loss of generality, we assume that $m_1\geq B_1+1$. By the assumption of the lemma, there exists a sum $F'=a_1P_{m_1'}+\cdots+a_rP_{m_r}$ such that $t(F')\leq t_{\max}$ and $\left\lfloor t(F')/a_k\right\rfloor+3\leq B_k$ for any $1\leq k\leq r$. Using Lemma \ref{thm::observation}, we see that $t(F)\leq t(F')\leq t_{\max}$ and it follows that for any $1\leq k\leq r$, we have
$$
\left\lfloor\frac{t(F)}{a_k}\right\rfloor+3\leq\left\lfloor\frac{t(F')}{a_k}\right\rfloor+3\leq B_k,
$$
as desired. Next we prove the induction step. Assuming the lemma and the inequality hold for $s=k-1$, we prove them for $2\leq s=k\leq r$. Without loss of generality, we may assume that $m_1\geq B_1+1$. Then we choose any sum $F'=a_1P_{m_1'}+\cdots+a_rP_{m_r}$ with $A_1\leq m_1'\leq B_1$. By the inductive hypothesis, we have $t(F')\leq t_{\max}$ and 
$$
\left\lfloor\frac{t(F')}{a_k}\right\rfloor+3\leq B_k
$$
for any $1\leq k\leq r$. Again by Lemma \ref{thm::observation}, we can conclude that $t(F)\leq t(F')\leq t_{\max}$ and the sum $F$ satisfies the inequality for any $1\leq k\leq r$. This finishes the proof.
\end{proof}

Next we prove upper bounds on the truants of non-leaf nodes of depth $3$ and depth $4$. 

\begin{proposition}
\label{thm::truant3}
Excluding a set of $3052$ nodes of depth $3$ with $m_1,m_2\leq161$ and $m_3\leq78$, any other node $F$ of depth $3$ has truant $t(F)\leq644$. In particular, a ternary sum is universal only if it is among the $3052$ excluded nodes.
\end{proposition}
\begin{proof}
With $A_1=A_2=A_3=3$, we search for the integers $B_1,B_2,B_3$ satisfying the assumptions of Lemma \ref{thm::boundingtruant} for each choice of integers $a_1,a_2,a_3$. By a computer program, we can verify that the integers $A_1=A_2=A_3=3$ and $B_1=B_2=161,B_3=78$ satisfy the assumptions for any possible choice of integers $a_1,a_2,a_3$ of nodes in the escalator tree $T_{\infty}$. By a computer program to check the representations by any node of depth $3$ with parameters $A_i\leq m_i\leq B_i$ for any $1\leq i\leq 3$ up to $3000$, we find that the truant of any node is bounded above by $644$ except $3052$ nodes representing every positive integer up to $3000$. Hence, the proposition follows from Lemma \ref{thm::boundingtruant}.
\end{proof}
\begin{remark}
Since we can not verify the universality of all of the excluded nodes, we only have an implicit upper bound on the truants of these nodes and the truants of their children by Lemma \ref{thm::observation}, which causes the ineffectiveness of the finiteness theorems.
\end{remark}

\begin{proposition}
\label{thm::truant4}
For the nodes of depth $4$, we have the following facts:
\begin{enumerate}[leftmargin=*]
\item Suppose that $m_i\leq1883$ for at least three indices $\{1\leq i\leq4\}$. The truant of a non-leaf node $a_1P_{m_1}+a_2P_{m_2}+a_3P_{m_3}+a_4P_{m_4}$ is bounded above by an implicit constant.
\item Suppose that $m_i\leq1883$ for at most two indices $\{1\leq i\leq4\}$. Excluding a set of finitely many potentially universal nodes, any other node $a_1P_{m_1}+a_2P_{m_2}+a_3P_{m_3}+a_4P_{m_4}$ is not universal and the truant is bounded above by $1880$.
\end{enumerate}
\end{proposition}
\begin{proof}
(1) Set $M=1883$. For any non-leaf node of depth $4$ with $m_i\leq M$ for at least three indices in $\{1\leq i\leq4\}$, the truant is bounded by an implicit constant by applying Lemma \ref{thm::observation}, including the potential children of excluded nodes of depth $3$.

\noindent (2) For any node of depth $4$ with $m_i\leq M$ for at most two indices in $\{1\leq i\leq4\}$, say $1\leq j,k\leq 4$, we can use a computer program to verify that the integers $A_i=3,B_i=M$ for $i=j,k$ and $A_i=B_i=M$ for $i\neq j,k$ satisfy the assumptions of Lemma \ref{thm::boundingtruant}. Thus, by a computer program, we see that the truant of any node with $m_i\leq M$ for at most two indices is bounded above by $1880$, excluding a set of finitely many potentially universal nodes.
\end{proof}
\begin{remark}
From the example in Remark \ref{rmk::ineffective}, in fact we exclude an infinite set of universal nodes with $m_i\leq 1883$ for exactly three indices.
\end{remark}

For dealing with nodes of higher depth, the naive enumeration is not practical. Thus we use analytic methods instead. To finish this section, we prove some properties that will be used later.

\begin{lemma}
\label{thm::threequarter}
Fix a prime number $p\geq3$. Suppose that $b_1,b_2\in\ZZ_p^{\times}$, $b_3\in\ZZ_p$, and $\alpha\in p\ZZ_p$ satisfying either $\alpha\in p^2\ZZ_p$ or $b_3\in\ZZ_p^{\times}$. If the polynomial $f(x_1,x_2,x_3)=b_1x_1^2+c_1x_1+b_2x_2^2+c_2x_2+\alpha(b_3x_3^2+c_3x_3)\in\ZZ_p[x_1,x_2,x_3]$ represents every class in $\ZZ/p^2\ZZ$, then 
$$
\legendre{-b_1b_2}{p}=1.
$$
\end{lemma}
\begin{proof}
First assume that $\alpha\in p^2\ZZ_p$. Replacing $x_i$ by $x_i-\frac{c_i}{2b_i}$ for $1\leq i\leq 2$, we see that $b_1x_1^2+b_2x_2^2$ represents every class in $\ZZ/p^2\ZZ$ by the assumption. In particular, there exist $w_1,w_2\in\ZZ_p$ such that $b_1w_1^2+b_2w_2^2\equiv p\pmod{p^2}$. Hence, we have $b_1w_1^2+b_2w_2^2\equiv0\pmod{p}$ such that either $w_1\in\ZZ_p^{\times}$ or $w_2\in\ZZ_p^{\times}$. Without loss of generality, assume that $w_1\in\ZZ_p^{\times}$. Then, we see that $(b_2w_2/w_1)^2\equiv-b_1b_2\pmod{p}$, which is the desired result.

Next assume that $b_3\in\ZZ_p^{\times}$ and $\alpha\in p\ZZ_p^{\times}$. Without loss of generality, we may assume that $\alpha=p$. Replacing $x_i$ by $x_i-\frac{c_i}{2b_i}$ for $1\leq i\leq 3$, we see that $b_1x_1^2+b_2x_2^2+pb_3x_3^2$ represents every class in $\ZZ/p^2\ZZ$ by the assumption. Thus there exist $w_1,w_2,w_3\in\ZZ_p$ such that $b_1w_1^2+b_2w_2^2+pb_3w_3^2\equiv kp\pmod{p^2}$ with $k\in\ZZ_p^{\times}$ such that $b_3k$ is a non-square modulo $p$. Hence we see that $b_1w_1^2+b_2w_2^2\equiv0\pmod{p}$ such that either $w_1\in\ZZ_p^{\times}$ or $w_2\in\ZZ_p^{\times}$, and the proof follows as in the previous case.
\end{proof}

\begin{lemma}
\label{thm::nondyadicgrowth}
Suppose that $p\geq3$ is a prime number. Let $F=a_1P_{m_1}+a_2P_{m_2}+a_3P_{m_3}$ be a node in $T_{\infty}$ such that $p\nmid a_1a_2(m_1-2)(m_2-2)$. If one of the following conditions holds
\begin{enumerate}[leftmargin=*]
\item $p^2\mid a_3$,
\item $p\mid a_3$, $p\nmid(m_3-2)$ and $t(F)\geq p^2$,
\end{enumerate}
then we have
$$
\legendre{-a_1a_2(m_1-2)(m_2-2)}{p}=1.
$$
\end{lemma}
\begin{proof}
If $p^2\mid a_3$, then the construction of $T_{\infty}$ implies that $t(a_1P_{m_1}+a_2P_{m_2})\geq p^2$. Hence $t(F)\geq p^2$ in both cases and thus $F$ represents every class in $\ZZ/p^2\ZZ$. The claim then follows by Lemma \ref{thm::threequarter}.
\end{proof}

\begin{lemma}
\label{thm::dyadicgrowth}
Let $F=P_{m_1}+a_2P_{m_2}+a_3P_{m_3}$ be a node in the tree $T_{\infty}$ with $4\mid m_1,m_2,m_3$. The following facts hold.
\begin{enumerate}[leftmargin=*]
\item If $a_2=1$ and $a_3=1$, then we have $t(F)\leq46$. If $t(F)\geq8$, then either $(m_1,m_2)=(4,8)$ or $(m_1,m_2,m_3)=(4,4,8)$.
\item If $a_2=1$ and $a_3=2$, then we have $t(F)\leq71$. If $t(F)\geq16$, then $(m_1,m_2)=(8,12)$ or $(m_1,m_2,m_3)\in \{(4,4,8),(4,12,4),(4,16,4),(8,8,4),(8,16,8)\}$.
\item If $a_2=1$ and $a_3=3$, then we have $t(F)\leq38$. If $t(F)\geq8$, then $(m_1,m_2)=(4,8)$.
\item If $a_2=2$, then we have $t(F)\leq 7$ if $a_3=2$, $t(F)\leq10$ if $a_3=3$, $t(F)\leq15$ if $a_3=4$, and $t(F)\leq12$ if $a_3=5$.
\end{enumerate}
\end{lemma}
\begin{proof}
This follows from a straightforward calculation that uses the fact that either $P_m(x)\in\{0,1,m-3,m,3m-8\}$ or $P_m(x)\geq 3m-3$. 
\end{proof}

\section{Representations by Shifted Lattices}
\label{sec::shiftedlattice}

For an integer $r\geq1$, a \textit{quadratic space} $V$ of dimension $r$ is a vector space over $\QQ$ of dimension $r$ equipped with a symmetric bilinear form $B\colon V\times V\to\QQ$. We can associate $V$ with a \textit{quadratic map} $Q\colon V\to\QQ$ defined by $Q(v)\coloneqq B(v,v)$ for any vector $v\in V$. A quadratic space $V$ is \textit{positive-definite} if $Q(v)>0$ for any non-zero vector $v\in V$. A \textit{lattice} $L$ in a quadratic space $V$ of dimension $r$ is a free $\ZZ$-submodule of $V$ of rank $r$. A lattice $L$ is \textit{integral} if $B(L,L)\subseteq\ZZ$. Set $L^{\dual}\coloneqq\{v\in V\mid2B(v,L)\subseteq\ZZ\}$. For an integral lattice $L$, the \textit{discriminant} is defined as $[L^{\dual}\colon L]$ and the \textit{level} is defined as the least positive integer $N$ such that $NQ(v)\in\ZZ$ for any vector $v\in L^{\dual}$.

A \textit{shifted lattice} $X$ in a quadratic space $V$ is a subset of $V$ of the form $L+\nu$, where $L$ is a lattice in $V$ and $\nu$ is a vector of $V$. A shifted lattice is \textit{positive-definite} if the underlying quadratic space $V$ is positive-definite. A shifted lattice $X$ is \textit{integral} if $B(X,X)\subseteq \ZZ$. The \textit{rank} of a shifted lattice $X=L+\nu$ is the rank of $L$ as a free $\ZZ$-module. For an integral shifted lattice $X$, a \textit{base lattice} of a shifted lattice $X$ is an integral lattice $L_X$ containing $X$. An integral shifted lattice always admits a base lattice because the lattice generated by $X$ is a base lattice of $X$. The \textit{discriminant}, the \textit{level}, and the \textit{conductor} of an integral shifted lattice $X$ relative to a base lattice $L_X$ of $X$ are defined as the discriminant of $L_X$, the level of $L_X$, and the least positive integer such that $ML_X\subseteq L$. For a shifted lattice $X=L+\nu$ of conductor $M$ and $a,d\in\ZZ$ with $ad\equiv 1\pmod{M}$, we let $\shiftinverse{X}{d}$ be the shifted lattice $L+a\nu$. Finally, we say that an element $n\in\QQ$ is \textit{represented} by a shifted lattice $X$ or equivalently the shifted lattice $X$ \textit{represents} an element $n\in\QQ$ if there is a vector $v\in X$ such that $Q(v)=n$. Such a vector $v\in X$ is called a \textit{representation} of $n$ by $X$. The \textit{number of representations} of $n$ by $X$ is denoted by $r_X(n)$.

Let $F$ be a sum of generalized polygonal numbers. In Section \ref{sec::Eisenstein}, we will construct a shifted lattice $X$ together with two integers $\mu\geq1,\rho\in\ZZ$ such that $r_F(n)=r_X(\mu n+\rho)$ for any integer $n\in\ZZ$. Therefore, to study the representations by $F$, it is equivalent to study the representations by the shifted lattice $X$. Let $\HH=\{\tau\in\CC\mid\imag(\tau)>0\}$ be the complex upper half-plane. For a positive-definite integral shifted lattice $X$, the theta series $\Theta_X\colon\HH\to\CC$ associated to $X$ is defined as
$$
\Theta_X(\tau)\coloneqq\sum_{v\in X}e^{2\pi i\tau Q(v)}=\sum_{n\in\ZZ}r_X(n)e^{2\pi i\tau n}.
$$
By \cite[Proposition 2.1]{shimura1973modular} and the theory of modular forms, the theta series $\Theta_X$ is a modular form, which splits into a sum of an Eisenstein series $E_X$ and a cusp form $G_X$. Let $a_{E_X}(n)$ and $a_{G_X}(n)$ denote the $n$-th Fourier coefficient of $E_X$ and $G_X$ for any integer $n\geq0$, respectively. We have 
\begin{equation}\label{eqn:rXnsplit}
r_X(n)=a_{E_X}(n)+a_{G_X}(n).
\end{equation}
By estimating the Fourier coefficients $a_{E_X}(n)$ and $a_{G_X}(n)$, we can show that $r_X(n)$ is positive.

To estimate the Eisenstein part, we apply Siegel's analytic theory of quadratic forms. Note that historically Siegel's papers \cite{siegel1935uber,siegel1936uber,siegel1937uber,siegel1951indefinite,siegel1951indefinite2} did not cover the cases of positive-definite shifted lattices. However, Siegel's results were later generalized to include positive-definite shifted lattices, for example, see \cite{van1949theory,weil1965,shimura2004inhomogeneous}.

The main tool is the Siegel--Minkowski formula, which interprets the Fourier coefficient of the Eisenstein part in terms of a product of local densities. We shall use the formulation given in \cite[(1.15)]{shimura2004inhomogeneous}. Suppose that the base lattice of the shifted lattice $X$ is denoted by $L_X$ and the quadratic map is denoted by $Q$. For $r\geq2$, we have
\begin{equation}
\label{eqn::siegelminkowski}
a_{E_X}(n)=\frac{c_r(2\pi)^{\frac{r}{2}}n^{\frac{r}{2}-1}}{[L_X^{\dual}\colon L_X]^{\frac{1}{2}}\Gamma(\frac{r}{2})}\prod_{p}\beta_p(n;X),
\end{equation}
where $\Gamma(x)$ is the usual Gamma function and we have $c_r\coloneqq1$ if $r\neq2$ and $c_2=\frac{1}{2}$; the quantities $\beta_p(n;X)$ are the local densities of $X$, which is defined as follows. Let $V_p$, $L_{X,p}$, and $X_p$ be the localizations of the quadratic space $V$, the base lattice $L_X$, and the shifted lattice $X$ at $p$, respectively. We choose the unique Haar measures $\mathrm{d}v$ on $V_p$ and $\mathrm{d}\sigma$ on $\QQ_p$ such that 
$$
\int_{L_{X,p}}\mathrm{d}v=\int_{\ZZ_p}\mathrm{d}\sigma=1.
$$
For any positive integer $n$ and any prime number $p$, the \textit{local density} $\beta_p(n;X)$ of a shifted lattice $X$ in the quadratic space $V$ is defined as the integral
$$
\beta_p(n;X)\coloneqq\int_{\QQ_p}\int_{X_p}e_p\big(\sigma(Q(v)-n)\big)\mathrm{d}v\mathrm{d}\sigma.
$$
where the function $e_p\colon\QQ_p\to\CC$ is defined as follows. For any $p$-adic number $\alpha\in\QQ_p$, we define $e_p(\alpha)\coloneqq e^{-2\pi ia}$ for some rational number $a\in\bigcup_{t=1}^{\infty}p^{-t}\ZZ$ such that $\alpha-a\in\ZZ_p$. To estimate the Fourier coefficient of the Eisenstein using (\ref{eqn::siegelminkowski}), it is clear that we have to evaluate the local densities, which is the main goal in the coming section.

\section{Formulae for local densities and a lower bound on the Eisenstein part}\label{sec::Eisenstein}

In this section, we derive explicit formulae for local densities $\beta_p(n;X)$ by similar arguments in \cite[Section 2]{yang1998} and apply them to bound the Eisenstein part.

\subsection{An Explicit Formula for Non-dyadic Local Densities}\label{sec:nondyadic}

Fix a prime number $p\geq3$. Suppose that we have an integral shifted lattice $X=L+\nu$ in a quadratic space $V$ with the associated quadratic map denoted by $Q$ and choose a base lattice $L_X$. By Jordan canonical form theorem \cite[Theorem 5.2.4 and Section 5.3]{kitaoka1999arithmetic}, there exists a basis $\{\eta_1,\ldots,\eta_r\}$ of the localization $L_p$ such that the Gram matrix of $L_p$ is a diagonal matrix with entries $A_1,\ldots,A_r\in\ZZ_p$ and $\nu=s_1\eta_1+\cdots+s_r\eta_r$ with rational numbers $s_1,\ldots,s_r\in\QQ_p$. Then, we have
\begin{equation}
\label{eqn::localdensitystartupp}
\begin{aligned}
\beta_p(n;X)&=\int_{\QQ_p}\int_{X_p}e_p\big(\sigma(Q(v)-n)\big)\mathrm{d}v\mathrm{d}\sigma\\
&=p^{-\ord_p([L_X\colon L])}\int_{\QQ_p}\int_{\ZZ_p^r}e_p\bigg(\sigma\bigg(Q\bigg(\sum_{i=1}^{r}(x_i+s_i)\eta_i\bigg)-n\bigg)\bigg)\mathrm{d}x_1\cdots\mathrm{d}x_r\mathrm{d}\sigma\\
&=p^{-\ord_p([L_X\colon L])}\int_{\QQ_p}\int_{\ZZ_p^r}e_p\bigg(\sigma\bigg(\sum_{i=1}^{r}\left(A_ix_i^2+2A_is_ix_i+A_is_i^2\right)-n\bigg)\bigg)\mathrm{d}x_1\cdots\mathrm{d}x_r\mathrm{d}\sigma,
\end{aligned}
\end{equation}
Thus, it boils down to evaluating the integral
$$
I_p(n;\phi)\coloneqq\int_{\QQ_p}\int_{\ZZ_p^r}e_p\big(\sigma\left(\phi(\mathbf{x})-n\right)\big)\mathrm{d}\mathbf{x}\mathrm{d}\sigma,
$$
where $\mathbf{x}\coloneqq(x_1,\ldots,x_r)$ and $\phi(\mathbf{x})=\sum_{i=1}^{r}(b_ix_i^2+c_ix_i)$ with $b_i,c_i\in\ZZ_p$ for $1\leq i\leq r$.

First, we evaluate the integral
$$
G_p(\sigma;b,c)\coloneqq\int_{\ZZ_p}e_p\left(\sigma(bx^2+cx)\right)\mathrm{d}x,
$$
for $\sigma\in\QQ_p$ and $b,c\in\ZZ_p$.

\begin{lemma}
\label{thm::gaussintegral}
Set $\mathfrak{t}\coloneqq\min(\ord_p(b),\ord_p(c))$. Suppose that $\sigma=up^{t}\in\ZZ_p$ such that $u\in\ZZ_p^{\times}$ and $t\in\ZZ$. Then we have
$$
G_p(\sigma;b,c)=
\begin{cases}
\displaystyle 1,&\text{if }t+\mathfrak{t}\geq0,\\
\displaystyle 0,&\text{if }t+\mathfrak{t}<0,\ord_p(b)>\ord_p(c),\\
\displaystyle e_p\left(-\frac{\sigma c^2}{4b}\right)\cdot\gamma_p(\sigma b)\cdot p^{\frac{t+\mathfrak{t}}{2}},&\text{if }t+\mathfrak{t}<0,\ord_p(b)\leq\ord_p(c),\\
\end{cases}
$$
where for $x=u_xp^{t_x}$ with $u_x\in\ZZ_p^{\times}$ and $t_x\in\ZZ$, we define
$$
\gamma_p(x)=
\begin{dcases}
\displaystyle1,&\text{if }t_x\text{ is even},\\
\displaystyle\varepsilon_p^3\cdot\legendre{u_x}{p},&\text{if }t_x\text{ is odd},
\end{dcases}
$$
and
$$
\varepsilon_p\coloneqq
\begin{dcases}
1, &\text{if }p\equiv1\pmod{4},\\
i, &\text{if }p\equiv3\pmod{4}.
\end{dcases}
$$
\end{lemma}
\begin{proof}
The first case is obvious. For the second case, replacing $x$ by $x+p^{-t-\mathfrak{t}-1}$ in the integral, we have $G_p(\sigma;b,c)=e_p(\sigma cp^{-t-\mathfrak{t}-1})\cdot G_p(\sigma;b,c)$, where the extra factors containing $b$ disappear because $e_p(x)=1$ for $x\in\ZZ_p$. Since $e_p(\sigma cp^{-t-\mathfrak{t}-1})\neq1$, the integral vanishes. For the last case, replacing $x$ by $x-\frac{c}{2b}$ in the integral, we have
$$
G_p(\sigma;b,c)=e_p\left(-\frac{\sigma c^2}{4b}\right)\cdot\int_{\ZZ_p}e_p(\sigma bx^2)\mathrm{d}x.
$$
Applying \cite[Lemma 2.1(1)]{yang1998}, we obtain the desired result.
\end{proof}

For the sake of simplicty, we define $\infty$ to be a formal symbol with the properties $t\leq\infty$ and $t+\infty=\infty+t=\infty$ for any integer $t\in\ZZ$. We set a convention that taking the minimum among an empty set outputs the formal symbol $\infty$.

\begin{theorem}
\label{thm::formulalocaldensityp}
Let $p\geq3$ be an odd prime number. Suppose that $n\in\ZZ_p$ and $\phi(\mathbf{x})=\sum_{i=1}^{r}(b_ix_i^2+c_ix_i)$ with $b_i,c_i\in\ZZ_p$ for $1\leq i\leq r$. For $1\leq i\leq r$, we define $t_i\coloneqq\min(\ord_p(b_i),\ord_p(c_i))$. We set
\begin{align*}
D_p\coloneqq\{1\leq i\leq r\mid\ord_p(b_i)>\ord_p(c_i)\},\\
N_p\coloneqq\{1\leq i\leq r\mid\ord_p(b_i)\leq\ord_p(c_i)\},
\end{align*}
and set $\td\coloneqq\min\{t_i\mid i\in D_p\}$. We further define
$$
\mathfrak{n}\coloneqq n+\sum_{i\in N_p}\frac{c_i^2}{4b_i}.
$$
If $\mathfrak{n}\neq0$, we assume that $\mathfrak{n}=\mathfrak{u}_np^{\mathfrak{t}_n}$ with $\mathfrak{u}_n\in\ZZ_p^{\times}$ and $\mathfrak{t}_n\in\ZZ$. Otherwise we set $\mathfrak{t}_n\coloneqq\infty$. For an integer $t\in\ZZ$, we define 
$$
\mathcal{L}_p(t)\coloneqq\{i\in N_p\mid t_i-t<0\text{ and odd}\},~\ell_p(t)\coloneqq|\mathcal{L}_p(t)|.
$$
Then, we have
$$
I_p(n;\phi)=1+\left(1-\frac{1}{p}\right)\sum_{\substack{1\leq t\leq\min(\td,\mathfrak{t}_n)\\ \ell_p(t)\text{ even}}}\delta_p(t)p^{\tau_p(t)}+\delta_p(\mathfrak{t}_n+1)\omega_pp^{\tau_p(\mathfrak{t}_n+1)},
$$
where for any integer $t\in\ZZ$, we define
$$
\delta_p(t)\coloneqq\varepsilon_p^{3\ell_p(t)}\prod_{i\in \mathcal{L}_p(t)}\legendre{u_i}{p},~\tau_p(t)\coloneqq t+\sum_{\substack{i\in N_p\\t_i<t}}\frac{t_i-t}{2},
$$
with $u_i\coloneqq p^{-\ord_p(b_i)}b_i\in\ZZ_p^{\times}$ and define
$$
\omega_p\coloneqq
\begin{dcases}
\displaystyle0,&\text{if }\mathfrak{t}_n\geq \td,\\
\displaystyle-\frac{1}{p},&\text{if }\mathfrak{t}_n<\td\text{ and }\ell_p(\mathfrak{t}_n+1)\text{ is even},\\
\displaystyle\varepsilon_p\legendre{\mathfrak{u}_n}{p}\frac{1}{\sqrt{p}},&\text{if }\mathfrak{t}_n<\td\text{ and }\ell_p(\mathfrak{t}_n+1)\text{ is odd}.
\end{dcases}
$$
\end{theorem}
\begin{proof}
By Lemma \ref{thm::gaussintegral}, we see that
\begin{align*}
I_p(n;\phi)&=\int_{\QQ_p}e_p(-\sigma n)\prod_{i=1}^{r}G_p(\sigma;b_i,c_i)\mathrm{d}\sigma\\
&=1+\sum_{1\leq t\leq \td}p^{t}\int_{\ZZ_p^{\times}}e_p(-p^{-t}\sigma n)\prod_{i=1}^{r}G_p(p^{-t}\sigma,b_i,c_i)\mathrm{d}\sigma\\
&=1+\sum_{1\leq t\leq \td}\delta_p(t)p^{\tau_p(t)}\int_{\ZZ_p^{\times}}\legendre{\sigma}{p}^{\ell_p(t)}e_p\Bigg(-p^{-t}\sigma\bigg(n+\sum_{\substack{i\in N_p\\t_i<t}}\frac{c_i^2}{4b_i}\bigg)\Bigg)\mathrm{d}\sigma.
\end{align*}
By definition, we have
$$
n+\sum_{\substack{i\in N_p\\t_i<t}}\frac{c_i^2}{4b_i}=\mathfrak{n}-\sum_{\substack{i\in N_p\\t_i\geq t}}\frac{c_i^2}{4b_i}\equiv \mathfrak{n}\pmod{p^t\ZZ_p}.
$$
Thus, by \cite[Lemma 2.4]{yang1998}, we can conclude that
\begin{align*}
I_p(n;\phi)&=1+\sum_{\substack{1\leq t\leq \td\\ \ell_p(t)\text{ even}}}\left(\chi_{p^t\ZZ_p}(\mathfrak{n})-\frac{1}{p}\chi_{p^{t-1}\ZZ_p}(\mathfrak{n})\right)\delta_p(t)p^{\tau_p(t)}+\sum_{\substack{1\leq t\leq \td\\ \ell_p(t)\text{ odd}}}\chi_{p^{t-1}\ZZ_p^{\times}}(\mathfrak{n})\delta_p(t)\omega_pp^{\tau_p(t)}\\
&=1+\left(1-\frac{1}{p}\right)\sum_{\substack{1\leq t\leq\min(\td,\mathfrak{t}_n)\\ \ell_p(t)\text{ even}}}\delta_p(t)p^{\tau_p(t)}+\delta_p(\mathfrak{t}_n+1)\omega_pp^{\tau_p(\mathfrak{t}_n+1)},
\end{align*}
where we denote $\chi_S(x)$ the indicator function of a subset $S\subseteq\QQ_p$ here and throughout.
\end{proof}

\subsection{An Explicit Formula for Dyadic Local Densities}\label{sec:dyadic}

The dyadic case $p=2$ is similar but slightly more complicated because the Jordan canonical form is not necessarily diagonal in general. Although a formula for diagonal cases is sufficient for our applications, we shall prove a formula in general for the completeness of the result. By Jordan canonical form theorem \cite[Theorem 5.2.5 and Section 5.3]{kitaoka1999arithmetic}, there exists a basis $\{\eta_1,\ldots,\eta_r\}$ of the localization $L_2$ such that $\nu=s_1\eta_1+\cdots+s_r\eta_r$ with rational numbers $s_1,\ldots,s_r\in\QQ_2$ and the Gram matrix of $L_2$ is a block-diagonal matrix with $1\times1$ blocks with entries $A_1,\ldots,A_{r_1}\in\ZZ_2$ and $2\times2$ blocks
$$
B_1
\begin{pmatrix}
0 & 1\\
1 & 0
\end{pmatrix}
,\ldots,B_{r_2}
\begin{pmatrix}
0 & 1\\
1 & 0
\end{pmatrix}
,C_1
\begin{pmatrix}
2 & 1\\ 
1 & 2
\end{pmatrix}
,\ldots,C_{r_3}
\begin{pmatrix}
2 & 1\\ 
1 & 2
\end{pmatrix},
$$
with $B_j,C_k\in\ZZ_2$ for $1\leq j\leq r_2$ and $1\leq k\leq r_3$ satisfying $r=r_1+2(r_2+r_3)$. Arguing in the same way as in the non-dyadic cases, we have
\begin{equation}
\label{eqn::localdensitystartup2}
\beta_2(n;X)=2^{-\ord_2([L_X\colon L])}I_2(n;\phi)
\end{equation}
and it boils down to the calculation of the following integral
\begin{align*}
I_2(n;\phi)\coloneqq\int_{\QQ_2}\int_{\ZZ_2^{r_1}\times\ZZ_2^{2r_2}\times\ZZ_2^{2r_3}}e_2\big(\sigma\left(\phi(\mathbf{x},\mathbf{y},\mathbf{z})-n\right)\big)\mathrm{d}\mathbf{x}\mathrm{d}\mathbf{y}\mathrm{d}\mathbf{z}\mathrm{d}\sigma,
\end{align*}
where $\mathbf{x}\coloneqq(x_1,\ldots,x_{r_1})$, $\mathbf{y}\coloneqq(y_{1,1},y_{1,2},\ldots,y_{r_2,1},y_{r_2,2})$, $\mathbf{z}\coloneqq(z_{1,1},z_{1,2}\ldots,z_{r_3,1},z_{r_3,2})$, and $\phi(\mathbf{x},\mathbf{y},\mathbf{z})\in\ZZ_2[\mathbf{x},\mathbf{y},\mathbf{z}]$ is a polynomial of the form
\begin{align*}
\phi(\mathbf{x},\mathbf{y},\mathbf{z})=\sum_{i=1}^{r_1}(b_ix_i^2+c_ix_i)&+\sum_{j=1}^{r_2}(b_j'y_{j,1}y_{j,2}+c_j'y_{j,1}+d_j'y_{j,2})\\
&+\sum_{k=1}^{r_3}\left(b_k''z_{k,1}^2+b_k''z_{k,1}z_{k,2}+b_k''z_{k,2}^2+c_k''z_{k,1}+d_k''z_{k,2}\right).
\end{align*}

First, we evaluate the following Gauss integrals
\begin{align*}
G_2(\sigma;b,c)\coloneqq&\int_{\ZZ_2}e_2\left(\sigma(bx^2+cx)\right)\mathrm{d}x,\\
G_2'(\sigma;b,c,d)\coloneqq&\int_{\ZZ_2^2}e_2\left(\sigma(by_1y_2+cy_1+dy_2)\right)\mathrm{d}y_1\mathrm{d}y_2,\\
G_2''(\sigma;b,c,d)\coloneqq&\int_{\ZZ_2^2}e_2\left(\sigma(b(z_1^2+z_2^2+z_1z_2)+cz_1+dz_2)\right)\mathrm{d}z_1\mathrm{d}z_2,
\end{align*}
for $\sigma\in\QQ_2$ and $b,c,d\in\ZZ_2$.

\begin{lemma}
\label{thm::gaussintegral21}
Set $\mathfrak{t}\coloneqq\min(\ord_2(b),\ord_2(c))$. Suppose that $\sigma=u2^{t}\in\ZZ_2$ such that $u\in\ZZ_2^{\times}$ and $t\in\ZZ$. Then we have
$$
G_2(\sigma;b,c)=
\begin{dcases}
1, &\text{if }t+\mathfrak{t}\geq0,\\
0, &\text{if }t+\mathfrak{t}<0,\ord_2(b)>\ord_2(c),\\
1, &\text{if }t+\mathfrak{t}=-1,\ord_2(b)=\ord_2(c),\\
0, &\text{if }t+\mathfrak{t}<-1,\ord_2(b)=\ord_2(c),\\
0, &\text{if }t+\mathfrak{t}=-1,\ord_2(b)<\ord_2(c),\\
e_2\left(\frac{uu_b}{8}-\frac{\sigma c^2}{4b}\right)\legendre{2}{uu_b}^{1+t+\mathfrak{t}}2^{\frac{1+t+\mathfrak{t}}{2}},&\text{if }t+\mathfrak{t}<-1,\ord_2(b)<\ord_2(c),
\end{dcases}
$$
where $u_b\coloneqq 2^{-\ord_2(b)}b\in\ZZ_2^{\times}$ and we extend the Legendre symbol to $\QQ_2$ via the Hilbert symbol by 
$$
\legendre{2}{x}\coloneqq
\begin{dcases}
(2,x)_2,&\text{if }x\in\ZZ_2^{\times},\\
0,&\text{if }x\in\QQ_2\setminus\ZZ_2^{\times}.
\end{dcases}
$$
\end{lemma}
\begin{proof}
Applying \cite[Lemma 4.3(1)]{yang1998}, the proof is essentially the same as the proof of Lemma \ref{thm::gaussintegral}, except that the cases when $\ord_2(b)=\ord_2(c)$ require a different treatment. Assume that $\ord_2(b)=\ord_2(c)$. It is clear that $G_2(\sigma;b,c)=1$ if $t+\mathfrak{t}\geq0$. If $t+\mathfrak{t}<0$, we put
$$
k\coloneqq\left\lfloor\frac{1-t-\mathfrak{t}}{2}\right\rfloor.
$$
Then, we have
\begin{align*}
G_2(\sigma;b,c)&=\frac{1}{2^k}\sum_{r\in\ZZ/2^k\ZZ}e_2(\sigma(br^2+cr))\int_{\ZZ_2}e_2(2^k\sigma x(c+2br))\mathrm{d}x\\
&=\frac{\chi_{\ZZ_2}(\sigma2^kc)}{2^k}\sum_{r\in\ZZ/2^k\ZZ}e_2(\sigma(br^2+cr)).
\end{align*}
Noticing that $\chi_{\ZZ_2}(\sigma2^kc)\neq0$ if and only if $t+\mathfrak{t}=-1$, we obtain the desired results.
\end{proof}

\begin{lemma}
\label{thm::gaussintegral22}
Set $\mathfrak{t}\coloneqq\min(\ord_2(b),\ord_2(c),\ord_2(d))$. Suppose that $\sigma=u2^{t}\in\ZZ_2$ such that $u\in\ZZ_2^{\times}$ and $t\in\ZZ$. Then we have
$$
G_2'(\sigma;b,c,d)=
\begin{dcases}
1, &\text{if }t+\mathfrak{t}\geq0,\\
0, &\text{if }t+\mathfrak{t}<0,\ord_2(b)>\min(\ord_2(c),\ord_2(d)),\\
e_2\left(-\frac{\sigma cd}{b}\right)2^{t+\mathfrak{t}},&\text{if }t+\mathfrak{t}<0,\ord_2(b)\leq\min(\ord_2(c),\ord_2(d)).
\end{dcases}
$$
\end{lemma}
\begin{proof}
The first case is obvious. For the second case, by the symmetry, we may assume that $\ord_2(c)\geq\ord_2(d)$. Therefore, we have
$$
G_2'(\sigma;b,c,d)=\int_{\ZZ_2}e_2(\sigma cy_1)\int_{\ZZ_2}e_2(\sigma(by_1+d)y_2)\mathrm{d}y_1\mathrm{d}y_2=\int_{\ZZ_2}e_2(\sigma cy_1)\chi_{\ZZ_2}(\sigma(by_1+d))\mathrm{d}y_1.
$$
The indicator function vanishes identically because $\ord_2(b)>\ord_2(d)$ and $t+\mathfrak{t}<0$. Therefore we obtain the desired result. For the third case, we apply the change of the variable $y_1\to y_1-\frac{d}{b}$ to the the second line of the above equation,
$$
G_2'(\sigma;b,c,d)=e_2\left(-\frac{\sigma cd}{b}\right)2^{t+\mathfrak{t}}\int_{\ZZ_2}e_2(\sigma2^{-t-\mathfrak{t}}cy_1)\mathrm{d}y_1=e_2\left(-\frac{\sigma cd}{b}\right)2^{t+\mathfrak{t}},
$$
as desired.
\end{proof}

\begin{lemma}
\label{thm::gaussintegral23}
Set $\mathfrak{t}\coloneqq\min(\ord_2(b),\ord_2(c),\ord_2(d))$. Suppose that $\sigma=u2^{t}\in\ZZ_2$ such that $u\in\ZZ_2^{\times}$ and $t\in\ZZ$. Then we have 
$$
G_2''(\sigma;b,c,d)=
\begin{dcases}
1, &\text{if }t+\mathfrak{t}\geq0,\\
0, &\text{if }t+\mathfrak{t}<0,\ord_2(b)>\min(\ord_2(c),\ord_2(d)),\\
(-1)^{t+\mathfrak{t}}e_2\left(\frac{\sigma(c^2+d^2-cd)}{3b}\right)2^{t+\mathfrak{t}},&\text{if }t+\mathfrak{t}<0, \ord_2(b)\leq\min(\ord_2(c),\ord_2(d)).
\end{dcases}
$$
\end{lemma}
\begin{proof}
The first case is obvious. For the second case, we may assume that $\ord_2(c)\geq\ord_2(d)$. Therefore by Lemma \ref{thm::gaussintegral21}, we have
\begin{align*}
G_2''(\sigma;b,c,d)&=\int_{\ZZ_2}e_2(\sigma(bz_1^2+cz_1))\int_{\ZZ_2}e_2(\sigma(bz_2^2+(bz_1+d)z_2))\mathrm{d}z_2\mathrm{d}z_1\\
&=\int_{\ZZ_2}e_2(\sigma(bz_1^2+cz_1))G_2(\sigma;b,bz_1+d)\mathrm{d}z_1=\chi_{\ZZ_2}(\sigma d)\int_{\ZZ_2}e_2(\sigma(bz_1^2+cz_1))\mathrm{d}z_1.
\end{align*}
The indicator function vanishes identically because $t+\mathfrak{t}<0$. So we obtain the desired results. For the third case, we apply the changes of variables $z_1\to z_1+\frac{d-2c}{3b}$ and $z_2\to z_2+\frac{c-2d}{3b}$. Then we have
\begin{align*}
G_2''(\sigma;b,c,d)&=e_2\left(\frac{\sigma(c^2+d^2-cd)}{3b}\right)\int_{\ZZ_2^2}e_2(\sigma b(z_1^2+z_2^2+z_1z_2))\mathrm{d}z_1\mathrm{d}z_2\\
&=(-1)^{t+\mathfrak{t}}e_2\left(\frac{\sigma(c^2+d^2-cd)}{3b}\right)2^{t+\mathfrak{t}},
\end{align*}
where the last equality follows from \cite[Lemma 4.4]{yang1998}.
\end{proof}

\begin{theorem}
\label{thm::formulalocaldensity2}
Suppose that $n\in\ZZ_2$ and $\phi(\mathbf{x},\mathbf{y},\mathbf{z})\in\ZZ_2[\mathbf{x},\mathbf{y},\mathbf{z}]$ is a polynomial of the form
\begin{align*}
\phi(\mathbf{x},\mathbf{y},\mathbf{z})=\sum_{i=1}^{r_1}(b_ix_i^2+c_ix_i)&+\sum_{j=1}^{r_2}(b_j'y_{j,1}y_{j,2}+c_j'y_{j,1}+d_j'y_{j,2})\\
&+\sum_{k=1}^{r_3}\left(b_k''z_{k,1}^2+b_k''z_{k,1}z_{k,2}+b_k''z_{k,2}^2+c_k''z_{k,1}+d_k''z_{k,2}\right).
\end{align*}
For $1\leq i\leq r_1$, we define $t_i\coloneqq\min(\ord_2(b_i),\ord_2(c_i))$. For $1\leq j\leq r_2$, we define $t_j'\coloneqq\min(\ord_2(b_j'),\ord_2(c_j'),\ord_2(d_j'))$. For $1\leq k\leq r_3$, we define $t_k''\coloneqq\min(\ord_2(b_k''),\ord_2(c_k''),\ord_2(d_k''))$. We set
\begin{align*}
D_2\coloneqq&\{1\leq i\leq r_1\mid\ord_2(b_i)>\ord_2(c_i)\},\\
E_2\coloneqq&\{1\leq i\leq r_1\mid\ord_2(b_i)=\ord_2(c_i)\},\\
N_2\coloneqq&\{1\leq i\leq r_1\mid\ord_2(b_i)<\ord_2(c_i)\},\\
D_2'\coloneqq&\{1\leq j\leq r_2\mid\ord_2(b_j')>\min(\ord_2(c_j'),\ord_2(d_j'))\},\\
N_2'\coloneqq&\{1\leq j\leq r_2\mid\ord_2(b_j')\leq\min(\ord_2(c_j'),\ord_2(d_j'))\},\\
D_2''\coloneqq&\{1\leq k\leq r_3\mid\ord_2(b_k'')>\min(\ord_2(c_k''),\ord_2(d_k''))\},\\
N_2''\coloneqq&\{1\leq k\leq r_3\mid\ord_2(b_k'')\leq\min(\ord_2(c_k''),\ord_2(d_k''))\},
\end{align*}
and set
$$
\td\coloneqq\min\Big\{\{t_i\mid i\in D_2\}\cup\{t_i+1\mid i\in E_2\}\cup\{t_j'\mid j\in D_2'\}\cup\{t_k''\mid k\in D_2''\}\Big\}.
$$
We further set
$$
\mathfrak{n}\coloneqq n+\sum_{i\in N_2}\frac{c_i^2}{4b_i}+\sum_{j\in N_2'}\frac{c_j'd_j'}{b_j'}+\sum_{k\in N_2''}\frac{c_k''^2+d_k''^2-c_k''d_k''}{3b_k''}.
$$
If $\mathfrak{n}\neq0$, assume that $\mathfrak{n}=\mathfrak{u}_n2^{\mathfrak{t}_n}$ such that $\mathfrak{u}_n\in\ZZ_2^{\times}$ and $\mathfrak{t}_n\in\ZZ$. Otherwise we set $\mathfrak{t}_n\coloneqq\infty$. For an integer $t\in\ZZ$, we define 
$$
\mathcal{L}_2(t)\coloneqq\{i\in N_2\mid t_i-t<0\text{ and odd}\},~\ell_2(t)\coloneqq|\mathcal{L}_2(t)|,
$$
and 
$$
\mathcal{L}_2''(t)\coloneqq\{k\in N_2''\mid t_k''-t<0\text{ and odd}\},~\ell_2''(t)\coloneqq|\mathcal{L}_2''(t)|.
$$
Then, we have
$$
I_2(n;\phi)=1+\left(1-\frac{1}{2}\right)\sum_{\substack{1\leq t\leq\min(\td,\mathfrak{t}_n+3)\\t\neq t_i+1\text{ for }i\in N_2}}\delta_2(t)2^{\tau_2(t)},
$$
where for any integer $t\in\ZZ$, we define
$$
\xi_2(t)\coloneqq\sum_{\substack{i\in N_2\\t_i+1<t}}u_i-2^{3-t}\mathfrak{n},
$$
with $u_i\coloneqq2^{-\ord_2(b_i)}b_i$, define
$$
\delta_2(t)\coloneqq
\begin{dcases}
(-1)^{\ell_2''(t)}\chi_{4\ZZ_2}(\xi_2(t))e_2\left(\frac{\xi_2(t)}{8}\right)\legendre{2}{\prod_{i\in\mathcal{L}_2(t-1)}u_i},&\text{if }\ell_2(t-1)\text{ is even},\\
(-1)^{\ell_2''(t)}\legendre{2}{\xi_2(t)\prod_{i\in\mathcal{L}_2(t-1)}u_i}\frac{1}{\sqrt{2}}, &\text{if }\ell_2(t-1)\text{ is odd},
\end{dcases}
$$
and define
$$
\tau_2(t)\coloneqq t+\frac{1}{2}\sum_{\substack{i\in N_2\\t_i+1<t}}(1+t_i-t)+\sum_{\substack{j\in N_2'\\t_j'<t}}(t_j'-t)+\sum_{\substack{k\in N_2''\\t_k''<t}}(t_k''-t).
$$
\end{theorem}
\begin{proof}
By Lemma \ref{thm::gaussintegral21}, Lemma \ref{thm::gaussintegral22}, and Lemma \ref{thm::gaussintegral23}, it is easy to see that
$$
I_2(n;\phi)=1+\sum_{\substack{1\leq t\leq\min(\td,\mathfrak{t}_n+3)\\t\neq t_i+1\text{ for }i\in N_2}}(-1)^{\ell_2''(t)}\legendre{2}{\prod_{i\in\mathcal{L}_2(t-1)}u_i}2^{\tau_2(t)}\int_{\ZZ_2^{\times}}\legendre{2}{\sigma}^{\ell_2(t-1)}e_2\left(\frac{\sigma\xi_2}{8}\right)\mathrm{d}\sigma.
$$
If $t>\mathfrak{t}_n+3$, then $\xi_2(t)\not\in\ZZ_2$. Then the integral vanishes by \cite[Lemma 4.3(2)]{yang1998}. Suppose that $t\leq \mathfrak{t}_n+3$ and $\ell_2(t-1)$ is even, we have
$$
\int_{\ZZ_2^{\times}}e_2\left(\frac{\sigma\xi_2(t)}{8}\right)\mathrm{d}\sigma=\frac{1}{2}\chi_{4\ZZ_2}(\xi_2(t))e_2\left(\frac{\xi_2(t)}{8}\right).
$$ 
Suppose that $t\leq \mathfrak{t}_n+3$ and $\ell_2(t-1)$ is odd. Again by \cite[Lemma 4.3(2)]{yang1998}, we have
$$
\int_{\ZZ_2^{\times}}\legendre{2}{\sigma}e_2\left(\frac{\sigma\xi_2(t)}{8}\right)\mathrm{d}\sigma=\frac{1}{2\sqrt{2}}\legendre{2}{\xi_2(t)},
$$
as desired.
\end{proof}

\subsection{Lower Bounds on the Eisenstein Part}
\label{sec::localdensitybound}

Let $F=\sum_{i=1}^{r}a_iP_{m_i}$ be a sum of generalized polygonal numbers. We are going to construct a shifted lattice $X$ together with integers $\mu\geq1,\rho\in\ZZ$ such that 
\begin{equation}
\label{eqn::splitting}
r_F(n)=r_X(\mu n+\rho)=a_{E_X}(\mu n+\rho)+a_{G_X}(\mu n+\rho),
\end{equation}
for any integer $n\in\ZZ$. The choice of such a shifted lattice $X$ with the integers $\mu$ and $\rho$ is not necessarily unique. Throughout this paper, we determine them as follows.

\begin{definition}
\label{thm::defineshiftedlattice}
Suppose that $F=\sum_{i=1}^{r}a_iP_{m_i}$ is a sum of generalized polygonal numbers. We set
$$
\Lambda\coloneqq\lcm(m_1-2,\ldots,m_r-2),~\mu\coloneqq8\Lambda,~\rho\coloneqq\sum_{i=1}^{r}\frac{a_i(m_i-4)^2\Lambda}{m_i-2}.
$$
For any $1\leq i\leq r$, we set
$$
\alpha_i\coloneqq\frac{a_i\Lambda}{m_i-2},~\mu_i\coloneqq2(m_i-2),~\rho_i\coloneqq 4-m_i.
$$
Let $V$ be a quadratic space of dimension $r$ with a basis $\{e_1,\ldots,e_r\}$ such that the Gram matrix with respect to the basis is a diagonal matrix with entries $\alpha_1,\ldots,\alpha_r$. \textit{The shifted lattice $X$ corresponding to the sum $F$} is defined to be
$$
X\coloneqq L+\nu\coloneqq\bigoplus_{i=1}^{r}\ZZ\langle\mu_ie_i\rangle+\sum_{i=1}^{r}\rho_ie_i,
$$
and the base lattice $L_X$ is defined to be the lattice generated by the basis $\{e_1,\ldots,e_r\}$. It is straightforward to verify that (\ref{eqn::splitting}) holds for the shifted lattice $X$. 
\end{definition}

In this subsection, we are going to bound the Fourier coefficient $a_{E_X}(\mu n+\rho)$ of the Eisenstein part in (\ref{eqn::splitting}) by Siegel-Minkowski formula and the explicit formulae for local densities. Plugging the data of the shifted lattice and the base lattice constructed in Definition \ref{thm::defineshiftedlattice} into (\ref{eqn::localdensitystartupp}) and (\ref{eqn::localdensitystartup2}), we have
\begin{equation}
\label{eqn::nodelocaldensitystartup}
\beta_p(\mu n+\rho;X)=p^{-\ord_p\left(\prod_{i}\mu_i\right)}I_p(8\Lambda n;4\Lambda\phi)=p^{-\ord_p\left(\prod_{i}\mu_i\right)+\ord_p(4\Lambda)}I_p(2n;\phi),
\end{equation}
where the quadratic polynomial is given by 
\begin{equation}
\label{eqn::localdensityphi}
\phi(\mathbf{x})\coloneqq\sum_{i=1}^{r}(a_i(m_i-2)x_i^2-a_i(m_i-4)x_i).
\end{equation}

It is relatively easy to bound the product of the local densities for $p\nmid2\prod_ia_i(m_i-2)$.

\begin{proposition}
\label{thm::almostallbound}
Suppose that $F=\sum_{i=1}^{r}a_iP_{m_i}$ is a sum of generalized polygonal numbers with $r\geq5$. Let $X,\mu,\rho$ be the shifted lattice and the integers constructed in Definition \ref{thm::defineshiftedlattice} corresponding to $F$. Then we have
$$
\prod_{p\nmid2\prod_{i}a_i(m_i-2)}\beta_p(\mu n+\rho;X)\sim 1,
$$
where $a(n)\sim b(n)$ means $b(n)\ll a(n)\ll b(n)$ for any functions $a,b\colon\NN\to\RR$.
\end{proposition}
\begin{proof}
Suppose that $p$ is a prime number such that $p\nmid2\prod_ia_i(m_i-2)$. Then we can eliminate the linear terms in the quadratic polynomial $\phi$ by applying a linear change of variables. Moreover by Jordan canonical form theorem \cite[Theorem 5.2.4 and Section 5.3]{kitaoka1999arithmetic}, we may further assume that 
$$
\phi(\mathbf{x})=x_1^2+\cdots+x_{r-1}^2+Dx_r^2
$$
where $D\coloneqq\prod_ia_i(m_i-2)\in\ZZ_p^{\times}$. Apply Theorem \ref{thm::formulalocaldensityp} to $I_p(2n;\phi)$, we have $t_i=0$ and $i\in N_p$ for every $1\leq i\leq r$ and $\tau_p(t)=\left(1-\frac{r}{2}\right)t$ for every integer $t\geq 1$. Hence it follows from a straightforward calculation that
$$
1-p^{-2}\leq I_p(2n;\phi)\leq1+p^{-2},
$$
provided that $r\geq5$. Finally, bounding against special values of the zeta function, we obtain the desired asymptotics.
\end{proof}

It remains to bound the local densities for prime numbers $p$ such that $p\mid2\prod_ia_i(m_i-2)$. 

\ifnew
\begin{proposition}
\label{thm::nondyadicbound}
Fix an odd prime number $p$. Suppose that $F=\sum_{i=1}^{r}a_iP_{m_i}$ is a node of depth $r\geq5$ in the escalator tree $T_{\infty}$. Let $X,\mu,\rho$ be the shifted lattice and the integers constructed in Definition \ref{thm::defineshiftedlattice} corresponding to $F$. For any $p$-adic integer $n\in\ZZ_p$, we have
$$
2(1-p^{-\frac{1}{2}})p^{-\max(\ord_p(a_i))}\leq\frac{\beta_p(\mu n+\rho;X)}{p^{\ord_p(\Lambda)-\ord_p\left(\prod_{i}(m_i-2)\right)}}\leq3,
$$
where $\Lambda\coloneqq\lcm(m_1-2,\ldots,m_r-2)$.
\end{proposition}
\begin{proof}
We apply Theorem \ref{thm::formulalocaldensityp} to $I_p(2n;\phi)$ with $\phi$ defined in (\ref{eqn::localdensityphi}). Following the notation in Theorem \ref{thm::formulalocaldensityp}, we have $t_i=\ord_p(a_i)$ for any $1\leq i\leq r$. Moreover, we have $i\in D_p$ if $p\mid(m_i-2)$ and $i\in N_p$ if $p\nmid(m_i-2)$. Reindexing by a permutation $\lambda\in S_r$, we may assume that $t_{\lambda(1)}\leq t_{\lambda(2)}\leq\cdots\leq t_{\lambda(r)}$ and $\lambda(i)\leq\lambda(j)$ if $t_{\lambda(i)}=t_{\lambda(j)}$ for any $1\leq i\leq j\leq r$. Now, we show that one of the following condition
\begin{enumerate}[leftmargin=*]
\item $\td=0$;
\item $\td\geq1$, $t_{\lambda(1)}=t_{\lambda(2)}=t_{\lambda(3)}=0$;
\item $\td=1$, $t_{\lambda(1)}=t_{\lambda(2)}=0,t_{\lambda(3)}=1$;
\item $\td\geq2$, $t_{\lambda(1)}=t_{\lambda(2)}=0,t_{\lambda(3)}=t_{\lambda(4)}=1$;
\item $\td\geq2$, $t_{\lambda(1)}=t_{\lambda(2)}=0,t_{\lambda(3)}=1,t_{\lambda(4)}\geq2$, $(-a_{\lambda(1)}a_{\lambda(2)}(m_{\lambda(1)}-2)(m_{\lambda(2)}-2)/p)=1$;
\item $\td\geq2$, $t_{\lambda(1)}=t_{\lambda(2)}=0,t_{\lambda(3)}=2$, $(-a_{\lambda(1)}a_{\lambda(2)}(m_{\lambda(1)}-2)(m_{\lambda(2)}-2)/p)=1$.
\end{enumerate}

Clearly $\td\geq0$ by definition. If $\td\geq0$, then condition (1) holds. So we may assume that $\td\geq1$. In this case, we have $p\nmid a_1a_2$. Indeed, by Table \ref{tbl::truant2}, we have $p\nmid a_1$, and while if $p\mid a_2$, then $p=3$ and $m_1=5$. Therefore $p\mid a_2$ implies that $\td=0$, which is a contradiction. Hence we have $p\nmid a_1a_2$, and then $\lambda(i)=i$ and $t_{\lambda(i)}=0$ for any $i\in\{1,2\}$. If $t_{\lambda(3)}=0$, then condition (2) holds. So we may assume that $t_{\lambda(3)}\geq 1$ and in particular $p\mid a_3$. In this case, Table \ref{tbl::truant2} implies that $3\leq p\leq 19$ and $\ord_p(a_3)\leq 2$. Assuming that $\ord_p(a_3)=1$, then we have $\lambda(3)=3$ and $t_{\lambda(3)}=1$. If $p\mid (m_3-2)$, then condition (3) holds, while if $t_{\lambda(4)}=1$, then either condition (3) or condition (4) holds. Now assume that $p\nmid (m_3-2)$ and $t_{\lambda(4)}\geq2$. Then $\td\geq2$ and by Lemma \ref{thm::nondyadicgrowth}, we see that condition (5) holds. So we have exhausted all of these cases with $\ord_p(a_3)=1$. Assuming that $\ord_3(a_3)=2$, we see that condition (6) holds by Lemma \ref{thm::nondyadicgrowth}.

Next we bound $I_p(2n;\phi)$ using Theorem \ref{thm::formulalocaldensityp}. If condition (1) holds, then we have $I_p(2n;\phi)=1$. If condition (2) holds, we see that $\lambda(i)\in N_p$ for $1\leq i\leq3$ and therefore $\tau_p(t)\leq-\frac{t}{2}$ for any integer $t\geq1$. Thus, we have
$$
|I_p(2n;\phi)-1|\leq\left(1-p^{-1}\right)\sum_{1\leq t\leq \mathfrak{t}_n}p^{-\frac{t}{2}}+|\omega_p|\cdot p^{-\frac{\mathfrak{t}_n+1}{2}}\leq\left(1-p^{-1}\right)\sum_{1\leq t\leq \mathfrak{t}_n+1}p^{-\frac{t}{2}}\leq p^{-\frac{1}{2}}+p^{-1}.
$$
Suppose that condition (3) holds next. If $\mathfrak{t}_n=0$, then we have $1-p^{-1}\leq I_p(2n;\phi)\leq 1+p^{-1}$. If $\mathfrak{t}_n\geq1$, then we have $|I_p(2n;\phi)-1|\leq1-p^{-1}$. 

If condition (4) holds, we have $\lambda(i)\in N_p$ for $1\leq i\leq 4$. It follows that $\tau_p(t)\leq1-t$ for any integer $t\geq1$. There are four different cases we need to consider separately. First, if $\td\neq\infty$ and $\mathfrak{t}_n<\td$, then we have
$$
|I_p(2n;\phi)-1|\leq\left(1-p^{-1}\right)\sum_{1\leq t\leq \mathfrak{t}_n}p^{1-t}+|\omega_p|\cdot p^{-\mathfrak{t}_n}\leq1-p^{-\mathfrak{t}_n+1}\leq1-p^{-\td}.
$$
Second, if $\td\neq\infty$ and $\mathfrak{t}_n\geq \td$, then the first sum runs over $1\leq t\leq \td$ and $\omega_p=0$. Similar arguments yield that $|I_p(2n;\phi)-1|\leq1-p^{-\td}$. Third, if $\td=\infty$ and $\mathfrak{t}_n\leq t_{\lambda(5)}$, then we have 
\begin{align*}
|I_p(2n;\phi)-1|\leq \left(1-p^{-1}\right)\sum_{1\leq t\leq \mathfrak{t}_n}p^{1-t} + p^{-\frac{1}{2}-\mathfrak{t}_n}=1-(1-p^{-\frac{1}{2}})p^{-\mathfrak{t}_n}\leq 1-(1-p^{-\frac{1}{2}})p^{-t_{\lambda(5)}}.
\end{align*}
Lastly, if $\td=\infty$ and $\mathfrak{t}_n\geq t_{\lambda(5)}+1$, then we have $\tau_p(t)\leq\frac{1}{2}-t$ for any integer $t\geq t_{\lambda(5)}+1$. Then, we have
\begin{align*}
|I_p(2n;\phi)-1|&\leq\left(1-p^{-1}\right)\bigg(\sum_{1\leq t\leq t_{\lambda(5)}}p^{1-t}+\sum_{t_{\lambda(5)}+1\leq t\leq \mathfrak{t}_n}p^{\frac{1}{2}-t}\bigg)+p^{-\mathfrak{t}_n-1}\\
&=1-p^{-t_{\lambda(5)}}+p^{-t_{\lambda(5)}-\frac{1}{2}}-p^{-\mathfrak{t}_n-\frac{1}{2}}+p^{-\mathfrak{t}_n-1}\leq1-\left(1-p^{-\frac{1}{2}}\right)p^{-t_{\lambda(5)}}.
\end{align*}

If condition (5) holds, then we see that $\delta_p(1)=1$. If $\mathfrak{t}_n=0$, then we have $I_p(2n,\phi)=1-p^{-1}$. If $\mathfrak{t}_n\geq1$, then we have $\tau_p(t)\leq\frac{1-t}{2}$ for any integer $t\geq1$ and we conclude that 
\begin{align*}
\left|I_p(2n;\phi)-2+p^{-1}\right|&\leq\left(1-p^{-1}\right)\sum_{2\leq t\leq \mathfrak{t}_n} p^{\frac{1-t}{2}} + p^{-\frac{1+\mathfrak{t}_n}{2}}\\
&=p^{-\frac{1}{2}}\left(1+p^{-\frac{1}{2}}\right)\left(1-p^{-\frac{\mathfrak{t}_n}{2}}\right)+ p^{-\frac{1+\mathfrak{t}_n}{2}}\leq p^{-\frac{1}{2}}+p^{-1}.
\end{align*}

Finally we suppose that condition (6) holds. Then we have $\delta_p(1)=\delta_p(2)=1$. If $\mathfrak{t}_n=0$, then we have $I_p(2n;\phi)=1-p^{-1}$. If $\mathfrak{t}_n=1$, then we have $|I_p(2n;\phi)-1|=1-p^{-1}$. If $\mathfrak{t}_n\geq2$ and $\td=2$, then we have $I_p(2n;\phi)= 3-2p^{-1}$. If $\mathfrak{t}_n\geq2$ and $\td\geq3$, then $\tau_p(t)\leq 1-\frac{t}{2}$ for $t\geq 3$ and we have 
\begin{align*}
\left|I_p(2n;\phi)-3+2p^{-1}\right|&\leq\left(1-p^{-1}\right)\sum_{3\leq t\leq \mathfrak{t}_n} p^{1-\frac{t}{2}} +p^{-\frac{\mathfrak{t}_n}{2}}\\
&=\left(p^{-\frac{1}{2}}+p^{-1}\right)\left(1-p^{1-\frac{\mathfrak{t}_n}{2}}\right)+p^{-\frac{\mathfrak{t}_n}{2}}\leq p^{-\frac{1}{2}}+p^{-1}.
\end{align*}
Combining the calculations with (\ref{eqn::nodelocaldensitystartup}), we obtain the desired results.
\end{proof}
\fi

\ifnew
\begin{proposition}
\label{thm::dyadicbound}
Suppose that $F=\sum_{i=1}^{r}a_iP_{m_i}$ is a node of depth $r\geq5$ in the escalator tree $T_{\infty}$. Let $X,\mu,\rho$ be the shifted lattice and the integers constructed in Definition \ref{thm::defineshiftedlattice} corresponding to $F$. For any $2$-adic integer $n\in\ZZ_2$, we have
$$
2^{-\max(\ord_2(a_i))-1}\leq\frac{\beta_2(\mu n+\rho;X)}{2^{\ord_2(\Lambda)-\ord_p\left(\prod_{i}2(m_i-2)\right)}}\leq5,
$$
where $\Lambda\coloneqq\lcm(m_1-2,\ldots,m_r-2)$.
\end{proposition}
\begin{proof}
We apply Theorem \ref{thm::formulalocaldensity2} to $I_2(2n;\phi)$ with $\phi$ defined in (\ref{eqn::localdensityphi}). Follow the notation in Theorem \ref{thm::formulalocaldensity2}. We have $t_i=\ord_2(a_i)+\varepsilon$ for any $1\leq i\leq r$ where 
$$
\varepsilon\coloneqq
\begin{dcases}
1,&\text{if }i\in D_2\cup N_2;\\
0,&\text{if }i\in E_2.
\end{dcases}
$$
Moreover, we have $i\in D_2$ if $m_i\equiv2\pmod{4}$, $i\in E_2$ if $m_i\equiv1\pmod{2}$, and $i\in N_2$ if $m_i\equiv0\pmod{4}$. Reindexing by a permutation $\lambda\in S_r$, we may assume that $t_{\lambda(1)}\leq t_{\lambda(2)}\leq\cdots\leq t_{\lambda(r)}$ and $\lambda(i)\leq\lambda(j)$ if $t_{\lambda(i)}=t_{\lambda(j)}$ for any $1\leq i\leq j\leq r$. Now, we show that one of the following condition:
\begin{enumerate}[leftmargin=*]
\item $\td\leq2$;
\item $\td\geq3$, $t_{\lambda(1)}=t_{\lambda(2)}=1,1\leq t_{\lambda(3)}\leq2$, $t_{\lambda(3)}\leq t_{\lambda(4)}\leq4$, $\lambda(3)\in N_2$, if $t_{\lambda(4)}=4$ then $\lambda(4)\in D_2\cup N_2$, if $t_{\lambda(4)}\geq3$ then $t_{\lambda(3)}=2$;
\item $\td=3$, $t_{\lambda(1)}=1,t_{\lambda(2)}=2$, $\lambda(2)\in N_2$; 
\item $\td\geq4$, $t_{\lambda(1)}=1,t_{\lambda(2)}=2,t_{\lambda(3)}=2,2\leq t_{\lambda(4)}\leq3$, $\lambda(4)\in N_2$;
\item $\td\geq4$, $t_{\lambda(1)}=1,t_{\lambda(2)}=2,t_{\lambda(3)}=3,3\leq t_{\lambda(4)}\leq4$, $\lambda(3)\in N_2$;
\item $\td\geq3$, $t_{\lambda(1)}=t_{\lambda(2)}=t_{\lambda(3)}=1,3\leq t_{\lambda(4)}\leq6$, if $t_{\lambda(4)}=6$ then $\lambda(4)\in D_2\cup N_2$, $\lambda(i)=i$ for $1\leq i\leq 3$, $(a_1,a_2,a_3)=(1,1,1),(1,1,3)$;
\item $\td\geq3$, $t_{\lambda(1)}=t_{\lambda(2)}=1,t_{\lambda(3)}=2,4\leq t_{\lambda(4)}\leq7$, if $t_{\lambda(4)}=4$ then $\lambda(4)\in E_2$, if $t_{\lambda(4)}=7$ then $\lambda(4)\in D_2\cup N_2$, $\lambda(i)=i$ for $1\leq i\leq 3$, $(a_1,a_2,a_3)=(1,1,2)$.
\end{enumerate}

If $\td\leq2$, then condition (1) holds. Suppose that $\td\geq3$ next. From this assumption, we have $t_{\lambda(1)}=1$ and $\lambda(1)=1$. Moreover, we have $1\leq a_2\leq 2$ from Table \ref{tbl::truant2}. Then it follows that $m_1,m_2\equiv0\pmod{4}$. We first deal with the case $a_2=2$. Since $t_2=2$, we have $t_{\lambda(2)}\in\{1,2\}$, depending on whether there exists $3\leq i\in N_2$ with $a_{i}$ odd, in which case $t_{\lambda(2)}=1$. If $t_{\lambda(2)}=1$, we can conclude that $\lambda(2)\neq2$. Since $2\leq a_3\leq 5$ by Table \ref{tbl::truant2}, $\lambda(2)=3$ when $a_3=3,5$. If $\lambda(2)=3$, we have $a_3\leq a_4\leq 12$ by Lemma \ref{thm::dyadicgrowth}. Thus, we have $1\leq t_{\lambda(3)}\leq 2$, $\lambda(3)\in N_2$, and $t_{\lambda(3)}\leq t_{\lambda(4)}\leq4$ and if $t_{\lambda(4)}=4$ then $\lambda(4)\in D_2\cup N_2$. Thus, condition (2) holds. If $\lambda(2)\neq3$, then $a_3=2,4$. In this case, condition (2) also holds. 

If $t_{\lambda(2)}=2$, then $\lambda(2)=2$ and $\lambda(2)\in N_2$ by our choice of $\lambda$. If $\td=3$, then condition (3) holds. If $\td\geq4$, we shall show that condition (4) or condition (5) holds. Since $m_1,m_2\equiv0\pmod{4}$, we have $a_3=2,4$ from Table \ref{tbl::truant2} and from our assumption $t_{\lambda(2)}=2$. If $a_3=2$, we have $t_{\lambda(3)}=2$, $\lambda(3)=3$, and $m_3\equiv0\pmod{4}$. By Lemma \ref{thm::dyadicgrowth}, we see that $2\leq a_4\leq7$. So, we have $2\leq t_{\lambda(4)}\leq3$ and $\lambda(4)\in N_2$. Therefore condition (4) holds. If $a_3=4$, we have $m_3\equiv0\pmod{4}$ since $\td\geq4$. It follows that $2\leq t_{\lambda(3)}\leq3$. If $t_{\lambda(3)}=2$, then condition (4) holds. If $t_{\lambda(3)}=3$, we see that $\lambda(3)=3$ and $\lambda(3)\in N_2$ by our choice of $\lambda$. Moreover, by Lemma \ref{thm::dyadicgrowth}, we have $4\leq a_4\leq 15$. Therefore we have $3\leq t_{\lambda(4)}\leq 4$. This shows that condition (5) holds. 

Next we deal with the case $a_2=1$. Then $t_{\lambda(2)}=1$ and $\lambda(2)=2$. From Table \ref{tbl::truant2}, we notice that $a_3\leq3$. Thus we have $1\leq t_{\lambda(3)}\leq 2$ and $\lambda(3)\in N_2$ by our choice of $\lambda$. If $t_{\lambda(3)}=1$, then Lemma \ref{thm::dyadicgrowth} (1) and (3) implies that $t_{\lambda(3)}\leq t_{\lambda(4)}\leq 6$, with $t_{\lambda(4)}=6$ only possible if $\lambda(4)=4\in D_2\cup N_2$ and $a_4=32$, giving condition (6) in that case. If $t_{\lambda(4)}\leq 2$, then condition (2) holds. If $3\leq t_{\lambda(4)}\leq 6$, then condition (6) holds. 

Now suppose that $t_{\lambda(3)}=2$. By Lemma \ref{thm::dyadicgrowth}, we have $2=t_{\lambda(3)}\leq t_{\lambda(4)}\leq 7$, with $t_{\lambda(4)}=7$ only possible if $t_4=64$ and $\lambda(4)=4\in D_2\cup N_2$, giving condition (7) in that case. If $2\leq t_{\lambda(4)}\leq 3$, then condition (2) holds. If $5\leq t_{\lambda(4)}\leq 7$, then condition (7) holds. Finally, if $t_{\lambda(4)}=4$, then either condition (2) holds or condition (7) holds, depending on whether $\lambda(4)\in D_2\cup N_2$ or $\lambda(4)\in E_2$.

Then we bound $I_2(2n;\phi)$ using the explicit formula from Theorem \ref{thm::formulalocaldensity2}. If $\td=0$, then $I_2(2n;\phi)=1$. Otherwise, we have $I_2(2n;\phi)=2+R_2(2n;\phi)$, where we define
$$
R_2(2n;\phi)\coloneqq\sum_{\substack{2\leq t\leq\min(\td,\mathfrak{t}_n+3)\\t\neq t_i+1\text{ for }i\in N_2}}\delta_2(t)2^{\tau_2(t)-1}.
$$
If condition (1) holds, then $R_2(2n;\phi)=0$. Suppose that condition (2) holds. If $\lambda(4)\in D_2\cup E_2$, then $|R_2(2n;\phi)|\leq1$. If $\lambda(4)\in N_2$ and $t_{\lambda(4)}\leq t_{\lambda(3)}+1$, we have $\delta_2(t)=0$ for $2\leq t\leq t_{\lambda(4)}+1$ and $\tau_2(t)-1\leq-t+\frac{t_{\lambda(3)}+t_{\lambda(4)}+4}{2}$ for $t\geq t_{\lambda(4)}+2$. If $\lambda(5)\in D_2\cup E_2$, then
$$
|R_2(2n;\phi)|\leq\sum_{t_{\lambda(4)}+2\leq t\leq t_{\lambda(5)}+1}2^{\tau_2(t)-1}\leq2-2^{-t_{\lambda(5)}+\frac{t_{\lambda(3)}+t_{\lambda(4)}+2}{2}}.
$$
If $\lambda(5)\in N_2$, then $\tau_2(t)-1\leq-\frac{3t}{2}+\frac{t_{\lambda(3)}+t_{\lambda(4)}+t_{\lambda(5)}+5}{2}$ for $t\geq t_{\lambda(5)}+2$. We have
\begin{align*}
|R_2(2n;\phi)|&\leq\sum_{t_{\lambda(4)}+2\leq t\leq t_{\lambda(5)}}2^{\tau_2(t)-1}+\sum_{t_{\lambda(5)}+2\leq t\leq\infty}2^{\tau_2(t)-1}\\
&\leq\max\left(0, 2-2^{-t_{\lambda(5)}+\frac{t_{\lambda(3)}+t_{\lambda(4)}+4}{2}}\right)+2^{-t_{\lambda(5)}+\frac{t_{\lambda(3)}+t_{\lambda(4)}+1}{2}}\\
&\leq2-2^{-t_{\lambda(5)}+\frac{t_{\lambda(3)}+t_{\lambda(4)}+2}{2}}.
\end{align*}
Similarly if $\lambda(4)\in N_2$ and $t_{\lambda(4)}=t_{\lambda(3)}+2$ then $t_{\lambda(3)}=2$ and $t_{\lambda(4)}=4$. Then, we see that,
$$
|R_2(2n;\phi)|\leq|\delta_2(4)|\cdot2^{\tau_2(4)-1}+\sum_{t_{\lambda(4)}+2\leq t\leq t_{\lambda(5)}+1}2^{\tau_2(t)-1}\leq2-2^{-t_{\lambda(5)}+\frac{t_{\lambda(4)}+4}{2}}.
$$
If condition (3) holds, we have $R_2(2n;\phi)=0$. If condition (4) or condition (5) holds, arguing as the case of condition (2), we see that
$$
|R_2(2n;\phi)|\leq2-2^{-t_{\lambda(5)}+\frac{t_{\lambda(4)}+4}{2}}.
$$
Suppose that condition (6) holds. If $t_{\lambda(4)}=3$ with $\lambda(4)\in D_2$, we have $|R_2(2n;\phi)|\leq1$. If $t_{\lambda(4)}=4$ with $\lambda(4)\in D_2$ or $t_{\lambda(4)}=3$ with $\lambda(4)\in E_2$, then 
$$
R_2(2n;\phi)=\delta_2(3)2^{\tau_2(3)-1}+\epsilon\cdot\delta_2(4)2^{\tau_2(4)-1},
$$
where $\epsilon\coloneqq0$ if $t_i=3$ and $i\in N_2$ for some $1\leq i\leq r$ and $\epsilon\coloneqq1$ otherwise. Notice that the right hand side is determined by $\epsilon$, $m_{\lambda(1)},m_{\lambda(2)},m_{\lambda(3)}\pmod{16}$ and $n\pmod{8}$. By a computer program and by using Lemma \ref{thm::dyadicgrowth} to get rid of certain choices of $m_{\lambda(1)},m_{\lambda(2)},m_{\lambda(3)}$, we can conclude that $-1\leq R_2(2n;\phi)\leq2$. If $5\leq t_{\lambda(4)}\leq 6$ with $\lambda(4)\in D_2$ or $4\leq t_{\lambda(4)}\leq 5$ with $\lambda(4)\in E_2$, we can argue in a similar manner to conclude that $-1\leq R_2(2n;\phi)\leq2$. If $\lambda(4)\in N_2$, we have $\tau_2(t)-1=-t+\frac{5+t_{\lambda(4)}}{2}$ for any integer $t\geq t_{\lambda(4)}+2$. Thus, for $3\leq t_{\lambda(4)}\leq6$, we have
$$
\left|\sum_{t_{\lambda(4)}+2\leq t\leq \td}2^{\tau_2(t)-1}\right|\leq1-2^{-t_{\lambda(5)}+\frac{t_{\lambda(4)}+3}{2}}.
$$
Combining with the calculation of the cases $\lambda(4)\in D_2\cup E_2$, we have
$$
2^{-t_{\lambda(5)}+\frac{t_{\lambda(4)}+3}{2}}-2\leq R_2(2n;\phi)\leq3-2^{-t_{\lambda(5)}+\frac{t_{\lambda(4)}+3}{2}}.
$$
The case of condition (7) is similar to that of condition (6). Using Lemma \ref{thm::dyadicgrowth}, we can prove that $-\frac{3}{2}\leq R_2(2n;\phi)\leq\frac{5}{2}$. Combining the calculations with (\ref{eqn::nodelocaldensitystartup}), we obtain the desired results.
\end{proof}
\fi

Now we can prove a lower bound on the Eisenstein part.

\begin{proposition}
\label{thm::nodelowerbound}
Suppose that $F=\sum_{i=1}^{r}a_iP_{m_i}$ is a node of depth $r\geq5$ in the escalator tree $T_{\infty}$. Let $X,\mu,\rho$ be the shifted lattice and the integers constructed in Definition \ref{thm::defineshiftedlattice} corresponding to $F$. For any real number $\varepsilon>0$, we have
$$
a_{E_X}(\mu n+\rho)\gg_r\frac{(\mu n+\rho)^{\frac{r}{2}-1}}{\Lambda^{r-1+\varepsilon}\cdot\Big(\prod_{i}a_i\Big)^{\frac{3}{2}+\varepsilon}},
$$
for any integer $n\geq0$, where $\Lambda\coloneqq\lcm(m_1-2,\ldots,m_r-2)$.
\end{proposition}
\begin{proof}
By our construction of the base lattice $L_X$ in Definition \ref{thm::defineshiftedlattice}, we have 
$$
[L_X^{\dual}\colon L_X]\ll_r\Lambda^r\prod_ia_i(m_i-2)^{-1}.
$$
Plugging this upper bound and Proposition \ref{thm::almostallbound} into (\ref{eqn::siegelminkowski}), we have
$$
a_{E_X}(\mu n+\rho)\gg_r\frac{(\mu n+\rho)^{\frac{r}{2}-1}}{\Lambda^{r-1}\cdot\Big(\prod_{i}a_i\Big)^{\frac{1}{2}}}\prod_{p\mid\prod_ia_i(m_i-2)}\beta_p(\mu n+\rho;X).
$$
Using Proposition \ref{thm::nondyadicbound} and Proposition \ref{thm::dyadicbound} to bound the remaining local densities and using Robin's bound \cite[Theorem 11]{robin1983} on the divisor function, it is not hard to obtain the desired bound.
\end{proof}

\section{Upper Bounds on the Cuspidal Part}\label{sec::cuspidal}

In this section, we prove upper bounds on the Fourier coefficient $a_{G_X}(\mu n+\rho)$ of the cuspidal part $G_X$ appearing in (\ref{eqn::splitting}). This requires the theory of modular forms. We will use the following congruence subgroups of $\SL_2(\ZZ)$:
$$
\Gamma_0(N)=\left\{\begin{pmatrix}a & b\\ c & d\end{pmatrix}\in\SL_2(\ZZ)\Biggm\vert c\equiv0\pmod{N}\right\},
$$
$$
\Gamma_1(N)=\left\{\begin{pmatrix}a & b\\ c & d\end{pmatrix}\in\Gamma_0(N)\Biggm\vert a\equiv1\pmod{N}\right\},
$$
and
$$
\Gamma(N)=\left\{\begin{pmatrix}a & b\\ c & d\end{pmatrix}\in\Gamma_1(N)\Biggm\vert b\equiv0\pmod{N}\right\}.
$$
Fix $k\in\frac{1}{2}\ZZ$. Assume that $\Gamma$ is a congruence subgroup such that $\Gamma\subseteq\Gamma_0(4)$ when $k\in\frac{1}{2}\ZZ-\ZZ$ and $\chi$ is a Dirichlet character. Letting $\theta(\tau)\coloneqq\sum_{n\in\ZZ} e^{2\pi i n^2 \tau}$ be the standard unary theta function of weight $\frac{1}{2}$, we say that a function $f:\HH\to\CC$ is a \begin{it}(holomorphic) modular form of weight $k$ and Nebentypus character $\chi$ (for the $\theta$-multiplier) on $\Gamma$\end{it} if for every $\gamma=\left(\begin{smallmatrix}a&b\\c&d\end{smallmatrix}\right)\in\Gamma$ we have 
\[
f(\gamma\tau)= \chi(d)\left(\frac{\theta(\gamma \tau)}{\theta(\tau)}\right)^{2k} f(\tau)
\]
and $f(\gamma\tau)$ grows at most polynomially in $\imag(\tau)$ as $\tau\to i\infty$; we furthermore call $f$ a \begin{it}cusp form\end{it} if $f(\gamma\tau)$ vanishes as $\tau\to i\infty$ for all $\gamma\in\Gamma$. We let $M_k(\Gamma,\chi)$ denote the space of modular forms of weight $k$ for the congruence subgroup $\Gamma$ with the Nebentypus $\chi$ and let $S_k(\Gamma,\chi)$ denote the subspace of cusp forms in $M_k(\Gamma,\chi)$.

In order to tie together the theory of modular forms and the algebraic theory of shifted lattices, we require the Siegel--Weil formula, which rewrites the Fourier coefficients of the Eisenstein series as the number of representations by the genus. To be precise, suppose that we have a shifted lattice $X$ in a quadratic space $V$. Let $O(V),O_{\AA}(V),O(X)$ denote the orthogonal group of $V$, the adelization of $O(V)$, and the subgroup of $O(V)$ consisting of isometries $\sigma$ such that $\sigma(X)=X$, respectively. By \cite[Lemma 1.2]{sun2016class}, $O(V)$ and $O_{\AA}(V)$ act on the set of shifted lattices in a quadratic space $V$. The \textit{genus} of a shifted lattice $X$ is the orbit of $X$ under the action of $O_{\AA}(V)$, denoted by $\genus(X)$. The \textit{isometry class} of a shifted lattice $X$ is the orbit of $X$ under the action of $O(V)$, denoted by $\class(X)$. By \cite[Corollary 2.3]{sun2016class}, the genus of $X$ splits into a finite number of isometry classes. So we take a set of representatives $X_1,\ldots,X_g$ of isometry classes in $\genus(X)$. The Siegel--Weil formula states that the Fourier coefficient $a_{E_X}(n)$ is a weighted sum of the number of representations by shifted lattices in $\genus(X)$ as follows,
$$
a_{E_X}(n)=\mass(X)^{-1}\sum_{i=1}^{g}\frac{r_{X_i}(n)}{|O(X_i)|},
$$
where the quantity $\mass(X)$ is the \begin{it}mass\end{it} of $X$, defined by
$$
\mass(X)\coloneqq\sum_{i=1}^{g}\frac{1}{|O(X_i)|}.
$$

Let $X$ be a positive-definite integral shifted lattice of rank $r$, level $N$, and conductor $M$. By \cite[Proposition 2.1]{shimura1973modular}, for any shifted lattice $Y$ in the genus of $X$, the theta series $\Theta_Y$ is a modular form in $M_{\frac{r}{2}}(\Gamma_0(M^2N)\cap\Gamma_1(M),\chi)$ for some Dirichlet character $\chi$. Therefore we see that $G_X=\in S_{\frac{r}{2}}(\Gamma_0(M^2N)\cap\Gamma_1(M),\chi)$. 

We decompose the cuspidal part $G_X$ by \cite[Theorem 2.5]{cho2018} as follows,
\begin{equation}
\label{eqn::chodecompostion}
G_X=\sum_{\psi}G_{X,\psi},
\end{equation}
where the sum runs over all Dirichlet character $\psi$ modulo $M$. For each Dirichlet character $\psi$ modulo $M$, we have $G_{X,\psi}\in S_{\frac{r}{2}}(\Gamma_0(M^2N),\chi\psi)$. We call $G_{X,\psi}$ the \textit{$\psi$-component} of the shifted lattice $X$. Thus, we reduce to the cases when the cusp forms are on the congruence subgroup $\Gamma_0(M^2N)$. 

We shall use the following approach to give upper bounds on the Fourier coefficients of the cusp forms. We can normalize a cusp form with respect to the Petersson norm, which is induced by the Petersson inner product. For two modular forms $f,g\in M_k(\Gamma,\chi)$ such that $fg$ is a cusp form, the Petersson inner product $\langle f,g\rangle_{\Gamma}$ is defined as
$$
\langle f,g\rangle_{\Gamma}\coloneqq\int_{\Gamma\backslash\HH}f(\tau)\overline{g(\tau)}v^k\frac{\mathrm{d}u\mathrm{d}v}{v^2},
$$
where we write $\tau=u+iv$ with $u,v\in\RR$. For a cusp form $f\in S_k(\Gamma,\chi)$, the \textit{Petersson norm} of $f$ is defined as
$$
\lVert f\rVert_{\Gamma}\coloneqq\sqrt{\langle f,f\rangle_{\Gamma}}.
$$
We note that while $\langle f,g\rangle_{\Gamma}$ depends on the choice of $\Gamma$, the normalization $\frac{\langle f,g\rangle_{\Gamma}}{[\SL_2(\ZZ)\colon\Gamma]}$ is independent of the choice of $\Gamma$; both normalizations appear in the literature and prove useful in different settings, so we keep the dependence on $\Gamma$ in our notation to avoid confusion with the normalization that is independent of $\Gamma$.

The next proposition reduces the problem of bounding the Fourier coefficients of a cusp form to bounding its Petersson norm.

\begin{proposition}
\label{thm::roughupperbound}
Fix $k\in\frac{1}{2}\ZZ$ such that $k\geq 2$. Let $N\geq1$ be an integer such that $4\mid N$ if $k\in\frac{1}{2}\ZZ-\ZZ$ and $\chi$ be a Dirichlet character modulo $N$. Let $f$ be a cusp form in $S_k(\Gamma_0(N),\chi)$. If $k\in\frac{1}{2}\ZZ-\ZZ$, for any real number $\varepsilon>0$, we have
$$
|a_f(n)|\ll_{\varepsilon,k}\lVert f\rVert_{\Gamma_0(N)}N^{\varepsilon} n^{\frac{k}{2}-\frac{1}{4}+\varepsilon}\left(\frac{(n,N)}{N}\right)^{\frac{1}{16}},
$$
for any integer $n\geq0$. If $k\in\ZZ$, for any real number $\varepsilon>0$, we have
$$
|a_f(n)|\ll_{\varepsilon,k}\lVert f\rVert_{\Gamma_0(N)}N^{\frac{1}{2}+\varepsilon}n^{\frac{k}{2}-\frac{1}{2}+\varepsilon}
$$
for any integer $n\geq0$. Moreover, both the implied constants are effective.
\end{proposition}
\begin{proof}
If $k\in\ZZ$, it is exactly \cite[Theorem 12]{schulze2018petersson}. If $k\in\frac{1}{2}\ZZ-\ZZ$, it follows from \cite[Theorem 1]{waibel2018fourier}. 
\end{proof}

By Proposition \ref{thm::roughupperbound}, it suffices to bound the Petersson norm of a cusp form. To do so, we adopt the method of Blomer in \cite{blomer2004uniform}. In the paper \cite{blomer2004uniform}, the method was designed for lattices. But it can be generalized to the cases of shifted lattices. Here we imitate the proof of \cite[Lemma 3.2]{kamaraj2022universal}.

\begin{proposition}
\label{thm::peterssonbound}
Let $G$ be the $\psi$-component of the cuspidal part of the theta series associated to an integral positive-definite shifted lattice $X$ of rank $r\geq4$, discriminant $D$, level $N$, and conductor $M$. For any real number $\varepsilon>0$, we have
$$
\lVert G\rVert_{\Gamma_0(M^2N)}\ll_{\varepsilon,r}\frac{(M^2N)^{r+1+\varepsilon}}{M},
$$
where the implied constant is effective.
\end{proposition}
\begin{proof}
Let $\mathcal{F}$ denote the usual fundamental domain of $\SL_2(\ZZ)$ acting on the upper half-plane $\HH$. Let $\{\gamma_1,\ldots,\gamma_s\}$ be a set of representatives of the cosets $\Gamma(M^2N)\backslash\SL_2(\ZZ)$. Let $\Gamma_{\infty}$ denote the parabolic subgroup of $\Gamma(M^2N)\backslash \SL_2(\ZZ)$ generated by $T\coloneqq\left(\begin{smallmatrix}1 & 1\\0 & 1\end{smallmatrix}\right)$, namely 
$$
\Gamma_{\infty}\coloneqq\left\{\begin{pmatrix}1 & 0\\0 & 1\end{pmatrix},\begin{pmatrix}1 & 1\\0 & 1\end{pmatrix},\ldots,\begin{pmatrix}1 & M^2N-1\\0 & 1\end{pmatrix}\right\}.
$$
Finally, let $\{\rho_1,\ldots,\rho_t\}$ be a set of representative of the double cosets $\Gamma(M^2N)\backslash\SL_2(\ZZ)\slash \Gamma_{\infty}$.

Set
$$
C_{M,N}\coloneqq\frac{[\SL_2(\ZZ)\colon\Gamma_0(M^2N)]}{[\SL_2(\ZZ)\colon\Gamma(M^2N)]}.
$$
By definition of the Petersson norm, we have
\begin{align*}
\lVert G\rVert_{\Gamma_0(M^2N)}^2=&~C_{M,N}\int_{\Gamma(M^2N)\backslash\HH}|G(\tau)|^2v^{\frac{r}{2}}\frac{\mathrm{d}u\mathrm{d}v}{v^2}=C_{M,N}\sum_{i=1}^{s}\int_{\gamma_i\mathcal{F}}|G(\tau)|^2v^{\frac{r}{2}}\frac{\mathrm{d}u\mathrm{d}v}{v^2}\\
=&~C_{M,N}\sum_{i=1}^{s}\int_{\mathcal{F}}\Big|G|_{k}\gamma_i(\tau)\Big|^2v^{\frac{r}{2}}\frac{\mathrm{d}u\mathrm{d}v}{v^2}=C_{M,N}\sum_{j=1}^{t}\int_{\bigcup_{\gamma\in \Gamma_{\infty}}\gamma\mathcal{F}}\Big|G|_{k}\rho_j(\tau)\Big|^2v^{\frac{r}{2}}\frac{\mathrm{d}u\mathrm{d}v}{v^2},
\end{align*}
where $k\coloneqq\frac{r}{2}$ and we denote
\[
f|_{k}\gamma(\tau)\coloneqq\left(\frac{\Theta(\gamma\tau)}{\Theta(\tau)}\right)^{-2k}f(\gamma\tau).
\]
The cusp associated to the matrix $\rho_j$ is $\rho_j(i\infty)$. Recall that every cusp of $\Gamma(M^2N)$ has cusp width $M^2N$, so the Fourier expansion of $G$ at $\rho_j(i\infty)$ may be written in the form (for some $a_j(n)\in\CC$)
$$
G|_{k} \rho_j(\tau)=\sum_{n=1}^{\infty}a_{j}(n)e^{\frac{2\pi i\tau n}{M^2N}}.
$$
Plugging in the Fourier expansions, we have
\begin{align*}
\int_{\bigcup_{\gamma\in \Gamma_{\infty}}\gamma\mathcal{F}}\Big|G|_{k}\rho_j(\tau)\Big|^2v^{\frac{r}{2}}&\frac{\mathrm{d}u\mathrm{d}v}{v^2}\leq \int_{\frac{\sqrt{3}}{2}}^{\infty} \int_{-\frac{1}{2}}^{M^2N-\frac{1}{2}}\Big|G|_{k}\rho_j(\tau)\Big|^2v^{\frac{r}{2}}\frac{\mathrm{d}u\mathrm{d}v}{v^2}\\
&=\sum_{n=1}^{\infty}\sum_{m=1}^{\infty}a_{j}(n)\overline{a_{j}(m)}\int_{\frac{\sqrt{3}}{2}}^{\infty}e^{-\frac{2\pi (n+m)v}{M^2N}}v^{\frac{r}{2}-2}\int_{-\frac{1}{2}}^{M^2N-\frac{1}{2}}e^{\frac{2\pi i(n-m)u}{M^2N}}\mathrm{d}u\mathrm{d}v\\
&=M^2N\sum_{n=1}^{\infty}|a_{j}(n)|^2\int_{\frac{\sqrt{3}}{2}}^{\infty}e^{-\frac{4\pi nv}{M^2N}}v^{\frac{r}{2}-2}\mathrm{d}v\\
&\ll_{r}(M^2N)^{\frac{r}{2}}\sum_{n=1}^{\infty}|a_j(n)|^2n^{1-\frac{r}{2}}\Gamma\left(\frac{r}{2}-1,\frac{2\sqrt{3}\pi n}{M^2N}\right)\\
&\ll_{r}(M^2N)^{2}\sum_{n=1}^{\infty}\frac{e^{-\frac{2\sqrt{3}\pi n}{M^2N}}}{n}|a_j(n)|^2,
\end{align*}
where $\Gamma(s,y)$ is the incomplte gamma function defined as
$$
\Gamma(s,y)\coloneqq\int_{y}^{\infty}e^{-t}t^{s-1}\mathrm{d}t
$$
for a complex variable $s\in\CC$ and a real variable $y\in\RR$ such that $\real(s)>0$ and $y>0$, and for the last step we use an asympotic upper bound $\Gamma(s,y)\ll_se^{-y}y^{s-1}$ as $y\to\infty$. Therefore, it remains to bound the absolute value of the Fourier coefficients $a_j(n)$ of $G$ at each cusp.

Then, by \cite[Theorem 2.5]{cho2018} and Siegel--Weil formula, we see that the $\psi$-component $G$ is constructed explicitly as follows,
\begin{equation*}
\begin{aligned}
G&=\frac{1}{\phi(M)}\sum_{d\pmod{M}}\overline{\psi(d)}\cdot\overline{\chi_{\Delta}(d)}\cdot G_X|\gamma_d\\
&=\frac{1}{\phi(M)}\sum_{d\pmod{M}}\frac{\overline{\psi(d)}}{\mass(\shiftinverse{X}{d})}\sum_{\class(Y)\subseteq\genus(\shiftinverse{X}{d})}\frac{\Theta_{\shiftinverse{X}{d}}-\Theta_Y}{|O(Y)|},
\end{aligned}
\end{equation*}
where $\gamma_d\in\Gamma_0(M^2N)$ is a matrix whose lower right entry is congruent to $d$ modulo $M$. Thus, by the triangle inequality and the linearity of the slash operator, it suffices to bound the Fourier coefficient of the theta series $\Theta_Y$ at the cusp $\rho_j(i\infty)$ for each class $\class(Y)\subseteq\genus(\shiftinverse{X}{d})$.

For any shifted lattice $Y=K+\mu\in\genus(\shiftinverse{X}{d})$, we are going to estimate the Fourier coefficient $a_j(n)$. There exists a base lattice $L_Y$ of $Y$ such that $Y$ is of level $N$ and conductor $M$ relative to $L_Y$. We choose a basis $\{e_1,\ldots,e_r\}$ of $L_Y$ such that $\{\mu_1e_1,\ldots,\mu_re_r\}$ is a basis of $K$ for integers $\mu_1,\ldots,\mu_r\geq1$ and $\mu=\rho_1e_1+\cdots+\rho_re_r$ with integers $\rho_1,\ldots,\rho_r\in\ZZ$. Let $A$ be the Hessian matrix of the bilinear form $B$ with respect to the basis $\{e_1,\ldots,e_r\}$. We need a modular transformation formula for the shifted lattice $Y$. Using Shimura's formulation of theta series \cite[(2.0)]{shimura1973modular} with the trivial spherical function $P$, we put
$$
\theta(\tau;A,h,N)\coloneqq\sum_{x\equiv h\pmod{N}}e^{2\pi i\tau\frac{x^{\mathsf{t}}Ax}{2N^2}},
$$
where $A$ is a $r\times r$ real symmetric matrix, $h$ is a vector of dimension $r$, and $N$ is a positive integer. The theta series $\Theta_Y$ is related to Shimura's theta series as follows,
\begin{align*}
\Theta_Y(\tau)=&\sum_{x_i\equiv\rho_i\mod{\mu_i}}e^{2\pi i\tau\frac{x^{\mathsf{t}}Ax}{2}}=\sum_{\substack{h\pmod{M}\\h_i\equiv\rho_i\pmod{\mu_i}}}\sum_{x\equiv h\pmod{M}}e^{2\pi i\tau\frac{x^{\mathsf{t}}Ax}{2}}\\
=&\sum_{\substack{h\pmod{M}\\h_i\equiv\rho_i\pmod{\mu_i}}}\theta(\tau;M^2A,MNh,M^2N).
\end{align*}
It is easy to verify that for Shimura's theta series $\theta(\tau;M^2A,MNh,M^2N)$, a stronger assumption than that in \cite[(2.1)]{shimura1973modular} is satisfied since the matrix $M^2A$ is even integral. Under such a stronger assumption, we can imitate the arguments in \cite[Pages 454--455]{shimura1973modular} to see that, for any matrix $\gamma=\left(\begin{smallmatrix}a & b\\c & d\end{smallmatrix}\right)\in\SL_2(\ZZ)$ with $c>0$, we have
$$
\Theta_Y(\gamma(\tau))=\frac{(-i(c\tau+d))^{\frac{r}{2}}}{c^{\frac{r}{2}}M^r\det(A)^{\frac{1}{2}}}\sum_{\substack{h\pmod{M}\\h_i\equiv\rho_i\pmod{\mu_i}}}\sum_{\substack{k\pmod{M^2N}\\Ak\equiv0\pmod{N}}}\Phi(MNh,M^2A,k,M^2N)\theta(\tau;M^2A,k,M^2N),
$$
where $\Phi(h,S,k,m)$ is defined as follows,
$$
\Phi(h,S,k,m)\coloneqq\sum_{\substack{g\pmod{cm}\\g\equiv h\pmod{m}}}e^{2\pi i\left(a\frac{g^{\mathsf{t}}Sg+2k^{\mathsf{t}}Sg+dk^{\mathsf{t}}Sk}{2cm^2}\right)}.
$$
By \cite[Page 15]{kamaraj2022universal}, we have $|\Phi(MNh,M^2A,k,M^2N)|\leq c^{\frac{r}{2}}\prod_{i=1}^{r}\gcd(M^2b_j,c)^{\frac{1}{2}}$, where $b_1,\ldots,b_r\in\ZZ$ are the diagonal entries of the Smith normal form of the Hessian matrix $A$. Hence, plugging it back to the transformation formula of $\Theta_Y$, we have
\begin{equation*}
\begin{aligned}
|a_{Y,\gamma}(n)|&\leq\frac{M^r}{\mu_1\cdots\mu_r}\sum_{\substack{k\pmod{M^2N}\\Ak\equiv0\pmod{N}}}\left|\{x\in\ZZ^r\mid x^{\mathsf{t}}Ax=2Nn,x\equiv k\pmod{M^2N}\}\right|\\
&\leq\frac{M^{r}}{\mu_1\cdots\mu_r}\left|\{x\in\ZZ^r\mid x^{\mathsf{t}}Ax\leq2Nn\}\right|\leq\frac{(M^2N)^{\frac{r}{2}}}{M}n^{\frac{r}{2}},
\end{aligned}
\end{equation*}
where $a_{Y,\gamma}(n)$ is the $n$-th Fourier coefficient of $\Theta_Y$ at the cusp $\gamma(i\infty)$ for a matrix $\gamma\in\SL_2(\ZZ)$. Since the explicit construction of the $\psi$-component $G$ is a weighted average of the theta series $\Theta_Y$, the same bound holds for the Fourier coefficient $a_j(n)$ for every $1\leq j\leq t$. Finally, plugging it back to the Petersson norm of $G$, we obtain
$$
\lVert G\rVert_{\Gamma_0(M^2N)}^2\ll_rC_{M,N}\sum_{j=1}^{t}\frac{(M^2N)^{r+2}}{M^2}\sum_{n=1}^{\infty}e^{-\frac{2\sqrt{3}\pi n}{M^2N}}n^{r-1}\ll_{r,\varepsilon}\frac{(M^2N)^{2r+2+\varepsilon}}{M^2}.
$$
Taking the square root, we obtain the desired result.
\end{proof}

Now we can state the main results of this section.

\begin{theorem}
\label{thm::cuspidalupperbound}
Let $G_X$ be the cuspidal part of the theta series associated to a positive-definite integral shifted lattice $X$ of rank $r\geq4$, level $N$, and conductor $M$. For any real number $\varepsilon>0$ and any integer $n\geq1$, we have the following upper bounds:
\begin{enumerate}[leftmargin=*]
\item If $r$ is odd, then we have
$$
|a_{G_X}(n)|\ll_{\varepsilon,r}\frac{\phi(M)(M^2N)^{r+1+\varepsilon}}{M}n^{\frac{r}{4}-\frac{1}{4}+\varepsilon}.
$$
\item If $r$ is even, then we have
$$
|a_{G_X}(n)|\ll_{\varepsilon,r}\frac{\phi(M)(M^2N)^{r+\frac{3}{2}+\varepsilon}}{M}n^{\frac{r}{4}-\frac{1}{2}+\varepsilon}.
$$
\end{enumerate}
Here the implied constants are effective and $\phi$ denotes Euler's totient function. 
\end{theorem}
\begin{proof}
It follows from Proposition \ref{thm::roughupperbound}, Proposition \ref{thm::peterssonbound}, and \eqref{eqn::chodecompostion}.
\end{proof}

Applying Theorem \ref{thm::cuspidalupperbound} to the shifted lattice constructed in Definition \ref{thm::defineshiftedlattice}, we obtain the following upper bounds.

\begin{proposition}
\label{thm::nodeupperbound}
Suppose that $F=\sum_{i=1}^{r}a_iP_{m_i}$ is a node of depth $r\geq4$ in the escalator tree $T_{\infty}$. Let $X,\mu,\rho$ be the shifted lattice and the integers constructed in Definition \ref{thm::defineshiftedlattice} corresponding to $F$. Set $\Lambda\coloneqq\lcm(m_1-2,\ldots,m_r-2)$. For any real number $\varepsilon>0$ and any integer $n\geq0$, the following inequalities hold:
\begin{enumerate}[leftmargin=*]
\item If $r$ is odd, then we have
$$
|a_{G_X}(\mu n+\rho)|\ll_{\varepsilon,r}\Big(\prod_{i}a_i\Big)^{r+1+\varepsilon}\Lambda^{3r+3+\varepsilon}(\mu n+\rho)^{\frac{r}{4}-\frac{1}{4}+\varepsilon}.
$$
\item If $r$ is even, we have
$$
|a_{G_X}(\mu n+\rho)|\ll_{\varepsilon,r}\Big(\prod_{i}a_i\Big)^{r+\frac{3}{2}+\varepsilon}\Lambda^{3r+\frac{9}{2}+\varepsilon}(\mu n+\rho)^{\frac{r}{4}-\frac{1}{2}+\varepsilon}.
$$
\end{enumerate}
In both cases, the implied constants are effective.
\end{proposition}
\begin{proof}
Suppose that the level and the conductor of the shifted lattice $X$ are denoted by $N$ and $M$, respectively. The conductor $M$ is bounded above by $2\Lambda$ and the level $N$ is bounded above by $4\Lambda\prod_ia_i$. Plugging them into Theorem \ref{thm::cuspidalupperbound}, we obtain the desired bounds.
\end{proof}

\section{Proof of the Main Results}\label{sec::main}

For any integers $\mathfrak{M}\geq\mathfrak{m}-2\geq1$, the existence of the constants $\Constant_{\mathfrak{M}}$ and $\Constant_{\mathfrak{m},\mathfrak{M}}$ follows from the following theorem.

\begin{theorem}
\label{thm::finitetree}
For any integer $\mathfrak{M}\geq1$, the tree $T_{\mathfrak{M}}$ is finite.
\end{theorem}
\begin{proof}
By induction on the depth, we see that the parameters $a_i$ and $m_i$ are bounded for any $i\geq1$ and there are finitely many nodes of a fixed depth in $T_{\mathfrak{M}}$. It suffices to show that the depth of the tree is finite. Using Proposition \ref{thm::nodelowerbound} and Proposition \ref{thm::nodeupperbound} to compare the Eisenstein part and the cuspidal part in (\ref{eqn::splitting}), the integers not represented by a node of depth $5$ is bounded below by a uniform constant only depending on $\mathfrak{M}$. Since any child node represents at least one more integer than its parent node, the depth of every subtree rooted at a node of depth $5$ must be finite. This implies that the depth of the tree is finite.
\end{proof}

Now we prove the main theorems.

\begin{proof}[Proof of Theorem \ref{thm::mainfinitenesstheorem1} and Theorem \ref{thm::mainfinitenesstheorem2}]
The existence of the constant $\Constant_{\mathfrak{M}}$ follows from Theorem \ref{thm::finitetree}. When $\mathfrak{M}=1$, the calculation in \cite{bosma2013triangular} reveals that $\Constant_{\mathfrak{M}}=8$. For the upper bound on $\Constant_{\mathfrak{M}}$, by Proposition \ref{thm::truant2}, Proposition \ref{thm::truant3}, and Proposition \ref{thm::truant4}, we can bound the truants of the non-leaf nodes of depth $r\leq4$ in $T$ by an implicit constant. So we arrive at depth $5$. The parameters $a_1,\ldots,a_5$ are also bounded by this implicit constant. Thus, by comparing the bounds in Proposition \ref{thm::nodelowerbound} and Proposition \ref{thm::nodeupperbound}, we see that any node of depth $5$ in $T_{\mathfrak{M}}$ represents every positive integer $n$ such that
$$
n\gg_{\varepsilon}\mathfrak{M}^{43+\varepsilon}.
$$
This implies that the truant of any non-leaf node of depth $5$ in $T_{\mathfrak{M}}$ is bounded above by the right hand side. Since any non-leaf node of depth $r\geq6$ will represent more integers than its ancestor of depth $5$, its truant must satisfy the same bound. This finishes the proof of Theorem \ref{thm::mainfinitenesstheorem1}. The proof of Theorem \ref{thm::mainfinitenesstheorem2} is clear from the above proof, because the ineffectiveness is due to the implicit upper bound for the truants of nodes of depth $r\leq4$.
\end{proof}

\begin{proof}[Proof of Theorem \ref{thm::mainfinitenesstheorem3}]
Since we have $\Constant_{\mathfrak{m},\mathfrak{M}}\leq\Constant_{\mathfrak{M}}$, the existence of the constant $\Constant_{\mathfrak{m},\mathfrak{M}}$ is clear for any integers $\mathfrak{m},\mathfrak{M}$ such that $\mathfrak{M}\geq\mathfrak{m}-2\geq1$. If $\mathfrak{m}\geq20$, any node of depth $r\leq4$ with $m_i\geq\mathfrak{m}$ for $1\leq i\leq r$ can not represent $16$. Thus the constant $\Constant_{\mathfrak{m},\mathfrak{M}}$ is effective and it satisfies the same bound in Theorem \ref{thm::mainfinitenesstheorem1} with an effective implied constant under the assumption. If $\mathfrak{m}\geq36$, any node of depth $r\leq5$ with $m_i\geq\mathfrak{m}$ for $1\leq i\leq r$ can not represent $32$. For nodes of depth $r=6$, we can use a stronger upper bound for the cuspidal part when $r$ is even. Therefore we have
$$
\Constant_{\mathfrak{m},\mathfrak{M}}\ll_{\varepsilon}\mathfrak{M}^{27+\varepsilon},
$$
as desired.
\end{proof}

\begin{proof}[Proof of Theorem \ref{thm::ternaryuniversal}]
The first statement follows from Proposition \ref{thm::truant3} and a computer search. For the second statement, there are $125$ cases that are proved to be universal. For the reference, see Table \ref{tbl::confirmedcases}. For the remaining $72$ cases, it is not hard to show that the genus of its corresponding shifted lattice contains only one class. Hence the cuspidal part is trivial by Siegel-Weil formula and the universality can be verified easily via local computations.
\end{proof}

\bibliographystyle{plain}
\bibliography{reference}

\end{document}